\documentclass{amsart}

\usepackage{amsmath,amsthm,amsfonts,amssymb,amscd}
\usepackage[mathcal]{eucal}
\usepackage{psfrag}
\usepackage{epstopdf}
\usepackage[dvips]{graphicx}
\usepackage{graphics,color}

\theoremstyle{plain}
\newtheorem{mainthm}{Theorem}
\newtheorem{theorem}{Theorem}[section]
\newtheorem{proposition}[theorem]{Proposition}

\newtheorem{lemma}[theorem]{Lemma}

\theoremstyle{definition}
\newtheorem{definition}[theorem]{Definition}
\newtheorem{remark}[theorem]{Remark}

\newcommand{\ph}{\mathcal{PH}_{\omega}}
\newcommand{\diff}{\operatorname{Diff}}
\newcommand{\dw}{\operatorname{Diff}^1_{\omega}(M)}
\newcommand{\per}{\operatorname{Per}}
\newcommand{\ang}{\operatorname{ang}}
\newcommand{\inte}{\operatorname{int}}
\newcommand{\cl}{\operatorname{cl}}

\newcommand{\N}{\mathbb{N}}
\newcommand{\R}{\mathbb{R}}

\newcommand{\eps}{\varepsilon}

\title[$C^1$-Genericity of Symplectic Diffeomorphisms]{$C^1$-Genericity of Symplectic Diffeomorphisms and Lower Bounds for Topological Entropy}

\author{Thiago Catalan and Vanderlei Horita}

\begin{document}

\begin{abstract}
There is a $C^1$-residual (Baire second class) subset $\mathcal{R}$ of
symplectic diffeomorphisms on $2d$-dimensional manifold, $d\ge 1$, such that
for every non-Anosov $f$ in $\mathcal{R}$ its topological entropy is lower
bounded by the supremum of the Lyapunov exponents of their hyperbolic periodic
points in the \emph{unbreakable central subbundle} (i.e., central direction with
no dominated splitting) of $f$.
The previous result deals with the fact that for $f$ in a residual set $\tilde{\mathcal{R}}$
of symplectic diffeomorphisms (containing $\mathcal{R}$) satisfies a trichotomy:
or $f$ is Anosov or $f$ is robustly transitive partially hyperbolic with
{\em unbreakable center} of dimension $2m$, $0 < m < d$, or $f$ has totally
elliptic periodic points dense on $M$.
In the second case, we also show the existence of a sequence of $m$-{\em elliptic}
periodic points converging to $M$.
Indeed, $\tilde{\mathcal{R}}$ contains an open and dense subset.
\end{abstract}

\maketitle

\let\thefootnote\relax
\let\thefootnote\relax
\let\thefootnote\relax\footnote{Work partially supported by CNPq, FAPEMIG,
FAPESP, and PRONEX.}

\section{Introduction}
\label{s.introduction}

The concept of topological entropy of a dynamical system provides information
about its complexity and it is invariant by conjugacy.
Topological entropy is a positive real number that, roughly, measures the rate
of exponential growth of the number of distinguishable orbits with finite but
arbitrary precision as time advances.
Precisely, let $(X, d)$ be a compact metric space and $f \colon X \to X$ be a
continuous map.
For each natural number $n$, we define the metric
$$
    d_n(x,y)=\max\{d(f^i(x),f^i(y)): 0\leq i\leq n\}.
$$
Note that, given any $\varepsilon > 0$ and $n \ge 1$, two points of $X$ are
$\varepsilon$-close with respect to this metric if their first $n$ iterates
are $\varepsilon$-close.
A subset $E$ of $X$ is said to be $(n, \varepsilon)$-{\em separated} if each
pair of distinct points of $E$ is at least $\varepsilon$ apart in the metric
$d_n$.
Denote by $N(n, \varepsilon)$ the maximum cardinality of an
$(n,\varepsilon)$-separated set.
The {\em topological entropy} of the map $f$ is defined by
$$
    h_{top}(f)=\lim_{\varepsilon\to 0} \left(\limsup_{n\to \infty}
    \frac{1}{n}\log N(n,\varepsilon)\right).
$$
Recall that the limit defining $h_{top}(f)$ always exists in the extended
real line (but could be infinite).

Lyapunov exponents are another useful tool to measure complexity
of a dynamical system. They are important constants to measure the asymptotic
behavior of dynamics in the tangent space level. Positive Lyapunov exponents
indicate orbital divergence and long-term unpredictability of a dynamical
system because the omnipresent uncertainty in determining its initial state
grows exponentially fast in time. In other words, Lyapunov exponents tell us
the rate of divergence of nearby trajectories.
More precisely, given a diffeomorphism $f$ over a manifold $M$, we say that
a real number $\lambda(x)$ is a {\it Lyapunov exponent} of $x\in M$ if there
exists a nonzero vector $v\in T_xM$ such that
$$
\lim_{n\rightarrow \pm\infty} \frac{1}{n}\log \|Df^n(x) \ v\|=\lambda(x).
$$

The main result in this work, Theorem~\ref{t.entropy}, relates these two
different ways to measure the complexity of a system.
Roughly, we provide lower bounds of the topological entropy for a class of
symplectic systems using Lyapunov exponents of the hyperbolic periodic points
in the central direction, i.e. taking $v$ above in the central subbundle.
Hence, we obtain an estimate to topological complexity of symplectic
systems via differential properties of its hyperbolic periodic points
for a class of symplectic diffeomorphisms.
By class we mean a residual subset of $\dw$ in the complement of Anosov
diffeomorphism set.
Let us make precise the statements.

\smallskip

We say that a diffeomorphism $f \colon M \to M$ is {\it partially hyperbolic}
if there exists a continuous $Df$-invariant splitting
$TM=E^s\oplus E^c\oplus E^u$ with non trivial extremal sub-bundles
$E^s$ and $E^u$, such that for every $x\in M$ and every $n$ large enough:
\begin{itemize}
\item the splitting is \emph{dominated}:
$$
\|Df^n|E^i(x)\| \ \|Df^{-n}| E^j(f^n(x))\|\leq \displaystyle\frac{1}{2},
\ \text{ for any } (i,j)=(s,c),\ (s,u), \ (c,u); \text{ and}
$$
\item the extremal subbundles are hyperbolic:
$$
 \|Df^n|E^s(x)\|\leq \displaystyle\frac{1}{2} \quad \text{ and }\quad
 \|Df^{-n}|E^u(x)\|\leq\displaystyle\frac{1}{2}.
$$
\end{itemize}
We say that a partially hyperbolic diffeomorphism $f \in \dw$ has
{\em un\-brea\-kable center bundle} if the center bundle $E^c$ has no dominated
sub-splitting for $f$.
If the center bundle $E^c$ is trivial then $f$ is hyperbolic, that is,
$f$ is an {\it Anosov diffeomorphism.}

Here, $(M^{2d},\omega)$ denotes a compact, connected, and boundaryless
symplectic manifold with dimension $2d$ and $\diff^1_{\omega}(M^{2d})$
denotes the set of $C^1$-diffeomorphisms on $(M^{2d},\omega)$ that
preserve the symplectic form $\omega$.
Recall that a partially hyperbolic symplectic diffeomorphism has even
dimensional unbreakable center.
So, we can split the set of partially hyperbolic diffeomorphisms in
subsets according to the dimension of their unbreakable center bundle.
We denote by $\ph^1(m) \subset \diff^1_\omega(M^{2d})$, $0 < m < d$, the
set of partially hyperbolic diffeomorphisms with $\dim(E^c) = 2m$.
For convenience, we denote by $\mathcal{PH}_{\omega}^1(0)=\mathcal{A}$
the subset of Anosov diffeomorphisms and by $\mathcal{PH}^1_{\omega}(d)$
the complement of the closure of the union of the set of all Anosov and
all partially hyperbolic diffeomorphisms.
Note that $\mathcal{PH}_{\omega}^1(i)$ and $\mathcal{PH}_{\omega}^1(j)$
are disjoint subsets for every distinct $0\leq i, j\leq d$.
Moreover, they split $\diff^1_{\omega}(M^{2d})$.

In order to state the first result, let us recall the definition of
elliptic periodic points.
Let $\per(f)$ be the set of periodic points of $f$ in
$\diff^1_\omega(M^{2d})$.
We say that $p \in \per(f)$ of period $k$ is an {\em $m$-elliptic
periodic point}, $0< m \le d$, if $Df^{k}(p)$ has exactly $2m$ non-real
and simple eigenvalues of modulus one, and all other eigenvalues has
modulus different from $1$.
Here, a $d$-elliptic periodic point is called {\it totally elliptic
periodic point}.

Recall that, as hyperbolic periodic points, $m$-elliptic periodic points
are robust for symplectic diffeomorphisms.
Also, if $f$ is a partially hyperbolic diffeomorphism and has an
$m$-elliptic periodic point then $\dim E^c$ must be larger than $2m$.
In particular, from the continuity of the partially hyperbolic splitting
and the robustness of $m$-elliptic periodic points for symplectic
diffeomorphisms it follows that every $f\in \ph^1(m)$ having an
$m$-elliptic periodic point belongs to the interior of $\ph^1(m)$.

In \cite{ArBC}, Arnaud, Bonatti, and Crovisier show that a generic
partially hyperbolic symplectic diffeomorphism $f\in \ph^1(m) \subset
\diff^1_\omega(M^{4})$, $1 \le m \le 2$, must have $m$-elliptic periodic
points dense on $M$.
They also conjectured that the same is true in $\diff^1_\omega(M^{2d})$ for
any $d \ge 1$.
The next result provides a positive answer to the conjecture.

\begin{mainthm}
\label{t.trichotomy}
There exists a residual subset $\tilde{\mathcal{R}} \subset \dw$, such that if
$f\in \tilde{\mathcal{R}}$ one of the following properties happens:
\begin{itemize}
\item[a)] $f$ is an Anosov diffeomorphism;
\item[b)] $f$ is a robust transitive partially hyperbolic diffeomorphism
in $\ph^1(m)$, for some $0 < m < d$, and there is a sequence of
$m$-elliptic periodic points converging to $M$ (in the Hausdorff
topology);
\item[c)] $f$ is not partially hyperbolic and has a sequence of totally
elliptic periodic points converging to $M$ (in the Hausdorff topology).
\end{itemize}
\end{mainthm}

In particular, generically the absence of totally elliptical periodic
points implies some level of (uniform/partial) hyperbolicity.

Let us recall some previous related results.
Newhouse in \cite{N2} shows that in the complement of the set of Anosov
symplectic diffeomorphisms (in $\diff_\omega^1(M^{2})$) there is a
residual subset of symplectic diffeomorphisms exhibiting $1$-elliptic
periodic points dense on $M$.
Arnaud \cite{A} prove the existence of an open and dense subset of
$\diff_\omega^1(M^{4})$ that or $f$ is Anosov, or $f$ is partially
hyperbolic, or $f$ has a totally elliptic periodic points on $M$.
These result was extended in two direction\,: in the first one,
Saghin and Xia \cite{SX} generalize to $\diff_\omega^1(M^{2d})$, for any
$d \ge 1$, this result also follows from Horita and Tahzibi \cite{HT}.
In the second direction, Arnaud, Bonatti, and Crovisier \cite{ArBC} has
a $4$-dimensional version of our Theorem~\ref{t.trichotomy}, as we
already mentioned.

\medskip

Now, let us address to topological entropy of non-Anosov maps.
We denote by $\tau(p,f)$ the period of a periodic point $p$ of $f$.
For $f \in \dw$ and $p$ a periodic point with some eigenvalue with modulus
different to $1$ we define
$$
\lambda_{min}(p,f) = \min \{ |\lambda| \colon  \lambda \text{ is an
eigenvalue of } Df^{\tau(p,f)}(p)\text{ with }|\lambda| >1\}.
$$
and if $f$ has hyperbolic periodic point we define
\begin{equation}
s(f) = \sup \left\{ \frac{1}{\tau(p,f)} \log \lambda_{min}(p,f) \colon \;
p \text{ hyperbolic periodic point of } f \right\}.
\end{equation}
Recall that generically (i.e. for a residual subset) symplectic
diffeomorphisms has hyperbolic periodic points (in fact they are dense).
So, $s(f)$ is well defined for $f$ in a residual subset of $\dw$.

Newhouse in \cite{N3} relates the topological entropy and $s(\cdot)$ for
certain non-Anosov symplectic diffeomorphisms on surfaces.
Recently, Catalan and Tahzibi in \cite{CT} generalize for symplectic
diffeomorphisms on any $2d$-dimensional manifold.

\begin{theorem} [\cite{N3} for $d=1$ and \cite{CT} for any $d\ge 1$)]
There is a residual subset $\mathcal{R} \subset \diff^1_\omega(M^{2d})$ of
$C^1$ symplectic  diffeomorphisms in $M$, such that for every non-Anosov
diffeomorphism $f \in \mathcal{R}$, we have
$$
h_{top}(f) \ge s(f).
$$
\label{NeCT}\end{theorem}

Moreover, Catalan and Tahzibi in \cite{CT} obtain stronger estimate for
symplectic diffeomorphisms on surface\,: for a generic non-Anosov surface
symplectic diffeomorphism $f$ one has $h_{top}(f)=s(f)$.

Roughly, in the proof of Theorem~\ref{NeCT}, the authors use the lack of
hyperbolicity to obtain the estimate.
Here we are able to use the lack of partial hyperbolicity getting better
estimates to topological entropy.
Let us be make this precise.

Let $A: V \rightarrow V$ be a linear operator defined on a vector space
$V$ and let $E\subset V$ be an $A$-invariant subspace.
We denote by $\sigma(A|E)$ the \emph{spectral  radius} of $A$ restrict
to $E$.
Hence, given $f\in \ph^1(m)$ in $0 < m \le d$, we define
$$
S_m(f)=\sup\left\{\frac{1}{\tau(p,f)}\log\sigma(Df^{\tau(p,f)}|E^c(p))
\colon p\text{ hyperbolic periodic point of } f \right\}.
$$

Let us remark that, according to Theorem~\ref{t.trichotomy}, for a
residual subset of $\diff_\omega^1(M^{2d})$, a non-Anosov diffeomorphism
$f$ belongs to a subset $\ph^1(m)$, for some $0 < m \le d$.
Indeed, it holds for an open and dense subset of $\tilde{\mathcal{R}} \setminus
\ph^1(0)$.
So, if $f \in \diff_\omega^1(M^{2d})$ is a generic non-Anosov
diffeomorphism then $S_m(f)$ is defined for some $m$.
Clearly, if  $0<m\leq d$ then $S_m(f)\geq s(f)$.
In fact, for $f \in \ph^1(1)$ the equality holds, i.e., $S_1(f)=s(f)$.
It is not difficult to show that $S_m(\cdot)$, $0<m\leq d$, is a lower
semicontinuous function as $s(.)$ is, see Section~\ref{s.bounds-entropy}.

\smallskip

The main result in this paper provides lower bounds for the topological
entropy of a non-Anosov symplectic diffeomorphism.

\begin{mainthm}
\label{t.entropy}
There exists a residual subset $\mathcal{R}\subset \diff^1_\omega(M^{2d})$,
$d\ge 1$, such that if $f\in\mathcal{R} \cap \ph^1(m)$, $0< m \le d$, then
$$
h_{top}(f)\geq S_m(f).
$$
\end{mainthm}

It is worth to remark that, generically, in the lack of any kind of
uniform or partial hyperbolicity, i.e., for generic $f$ in $\ph^1(d)$,
the previous result yields a lower bound to topological entropy in terms
of the supremum of the largest Lyapunov exponent of all hyperbolic
periodic points of $f$.

\medskip

The paper is organized as follows, in Section~\ref{s.preliminaries} we
recall and provide some useful perturbative results in the symplectic
scenario as connecting lemma, Franks Lemma, and linear systems with
transitions.
Section~\ref{s.proof_t.trichotomy} is devoted to prove
Theorem~\ref{t.trichotomy}.
Using periodic linear systems we show in Section~\ref{s.bounds-entropy},
Lemma~\ref{exp}, how to perturb a symplectic diffeomorphism in order to
find a nice periodic point, namely a diagonalizable periodic point, having
Lyapunov exponents close to the ones of a previous arbitrary periodic
periodic point.
Furthermore, in Proposition~\ref{p.diagonalizable}, we show that $S_m(f)$
can take account just diagonalizable hyperbolic periodic points and
using a technical result, Proposition~\ref{p.estimate}, we complete the
proof of Theorem~\ref{t.entropy}.
We obtain,
in Section~\ref{s.nice_invariant_manifolds}, intersections between
{\emph strong stable and unstable manifolds} of diagonalizable periodic
points with small angles, Lemma~\ref{lemma1} and Lemma \ref{lemma2}.
These lemmas are essential to prove Proposition~\ref{p.estimate} in
Section~\ref{section5}.

\medskip

We end up this section giving a sketch of the proof of the main theorems.
We point out what we should overcome from the technics used in \cite{N3}
and \cite{CT} in order to prove Theorem~\ref{t.entropy}.
Also, we put how Theorem~\ref{t.trichotomy} follows from the technics
developed in order to prove Theorem~\ref{t.entropy}.

Let us recall key points in the proof of Theorem~\ref{NeCT}.
An essential point is that in the symplectic scenario the Palis
conjecture is known, more precisely, symplectic diffeomorphisms either are
approximated by symplectic Anosov diffeomorphisms or by diffeomorphisms
exhibiting homoclinic tangencies, see Newhouse \cite{N2}.
Another essential point is also due to Newhouse that show how construct
perturbation of a symplectic surface diffeomorphism $f$ exhibiting a
homoclinic tangency for a hyperbolic periodic point $p$, in order to
create a basic hyperbolic set having topological entropy arbitrary close
to the (unique) positive Lyapunov exponent of $p$ for $f$.
This kind of perturbation is called here {\it snake perturbation}.
To prove Theorem~\ref{NeCT} for symplectic diffeomorphisms in higher
dimension, Catalan and Tahzibi developed higher dimensional snake
perturbations.
The hyperbolic basic set obtained after a snake perturbation, has
topological entropy close to the smallest positive Lyapunov exponent of a
periodic periodic point $p$.
This is because of the natural $Df^{\tau(p)}$-invariant dominated
splitting in $T_pM$ given by the eigenspaces of $Df^{\tau(p)}(p)$.
This is the reason that Theorem~\ref{NeCT} provide lower bounds for
topological entropy in terms of the smallest positive Lyapunov
exponents of all hyperbolic periodic points.

Hence, to prove Theorem~\ref{t.entropy} following the program of
Theorem~\ref{NeCT}, we need perturb the diffeomorphism to build up,
for a hyperbolic periodic point $p$ of a $f$, a basic hyperbolic sets
associated to it with entropy close to an other Lyapunov exponents of
$p$.
To overcome this point, we find an $f$-invariant symplectic submanifold
$D\subset M$ containing $p$, such that the stable and unstable manifolds
of $p$ inside $D$ has a non-transversal intersection in $D$, and thus we
use snake perturbations of $f$ inside $D$, to construct a basic hyperbolic
set having entropy close to the smallest Lyapunov exponent of $p$ for $f$
restrict to $D$.

One of the steps in order to find the submanifold $D$, is the existence
of a periodic point $p$ having all central eigenvalues equal to one, after
a perturbation.
This result is due to Horita and Tahzibi \cite{HT}.
Theorem~\ref{t.trichotomy} is a consequence of this fact.

\section{Preliminaries}
\label{s.preliminaries}
In this section we recall some techniques and provide results that we use
along the proofs of Theorems~\ref{t.trichotomy} and \ref{t.entropy}.
They encompass perturbations of linear symplectic transformation,
connection of invariant manifolds, and periodic symplectic linear systems
with transitions.

\subsection{Linear symplectic perturbations}
\label{ss.linear}
First, let us recall some basic facts about symplectic vector spaces.
Let $(V, \omega)$ be a symplectic vector space of dimension $2d$.
For any subspace $W \subset V$ we define its {\it symplectic orthogonal}
vector space as
$$
  W^{\omega} = \{  v \in V; \omega(v, w) = 0 \quad \text{for all}
  \quad w \in W \}.
$$
The subspace $W$ is called {\it symplectic} if $W^{\omega} \cap
W = \{0\}$.
We say $W$ is {\it isotropic} if $W \subset W^{\omega},$ that is
$\omega|(W\times W) =0$.
When $W = W^{\omega}$ we say that $W$ is a {\it Lagrangian}
subspace.

For a symplectic form $\omega$ in $V$, there is a symplectic basis
$\mathcal{B}=\{e_1,\ldots, e_{2d}\}$ of $V$ such that, with respect
to this basis, $\omega$ is in the standard form
$\omega=\sum_{i=1}^d de_i \wedge de_{i+d}$, i.e.,
$\omega(e_i, e_{d+i})=1$ and $\omega(e_i,e_j)=0$ if $j\neq d+i$,
for every $1\leq i\leq d$.
Now, if $J$ is the canonical map on $V$, with respect to $\mathcal{B}$,
such that $J^2=-Id$,  we say that a {\it linear map $A$ is symplectic}
if $A^*JA=J$.
In particular, $A$ is a symplectic map if, and only if, $A^*\omega=\omega$.
Note that, if we take an inner product on $V$ for which $\mathcal{B}$ is
an orthonormal basis, then $\omega(u,v)=<u ,\ Jv>$.

Given two vector subspaces $E$ and $E'$ of a vector space $V$ endowed with
a inner product, $\dim E = \dim E' = j$, we say that $E$ and $E'$ are
\emph{$\delta$-close} if there exists orthonormal basis
$\{e_1, \dots , e_j \}$ of $E$ and $\{e_1', \dots , e_j' \}$ of $E'$ such
that $\max\{ \|e_i - e_i' \| \colon 1 \le i \le j \} < \delta$, where
$\| \cdot \|$ is induced by the inner product.
For any pair of $E$ and $E'$ of vector subspaces of same dimension it is
trivial to find a linear isomorphism $A$ such that $A(E') = E$.
Moreover, if $E$ and $E'$ are close then $A$ can be choosen close to the
identity.
Next lemma asserts that if $E$ and $E'$ are close then $A$ can be
taken symplectic and preserving a complementary space of $E$.

In the reminder of this section, for sake of simplicity we denote
$V=(V,\omega)$ a symplectic vector space.

\begin{lemma}
\label{isotropic2}
Suppose $V=E\oplus F$, where $E$ is an isotropic subspace.
For any $\eps>0$, there exists $\delta>0$ such that if $W\subset V$ is an
isotropic subspace $\delta$-close to $E$, then there exists a symplectic
linear map $B$ on $V$ $\eps$-close to $Id$ such that $B(W)=E$ and
$B|F=Id_F$.
\end{lemma}

\begin{proof}
We use here coordinates $(x,y)$ in $V$ with respect to the decomposition
$V=E\oplus F$, and we fix an arbitrary norm $\| \cdot \|$ in $V$.

Since $W$ is close enough to $E$, there exists a linear map
$A \colon E \rightarrow F$  such that $W=\{(x,Ax) \colon x\in E\}$.
Moreover, given $\eps>0$ we can choose $\delta>0$ small enough such that
if $W$ is $\delta$-close to $E$, then $\|A\|<\eps$.
Thus if we define $j \colon E \rightarrow V$ by $j(x)=(x,A(x))$, since $W$
is an isotropic subspace, we have $j^*\omega=0$, where $j^*\omega$ is the
pull-back of the symplectic form $\omega$ by $j$.
Analogously, if $i \colon E\rightarrow V$ is the natural inclusion,
$i(x)=(x,0)$, we have $i^*\omega=0$.

Finally, if we define $B:V\rightarrow V$ by $B(x,y)=(x,y-A(x))$, then to
conclude the proof of the lemma we just need to show that $B$ is indeed
symplectic, since $\|B-Id\|\leq \|A\|<\eps$, $B(W)=E$ and $B|F=Id_F$.

Hence, let $\pi \colon V\rightarrow E$ be the projection on the first
coordinate, i.e.,  $\pi(x,y)=x$.
We rewrite $B$ as $B=Id+i\circ\pi-j\circ\pi$.
Therefore, we finally can verify that
$$
B^*\omega=\omega+\pi^*i^*\omega-\pi^*j^*\omega=\omega,
$$
where we use that $i^*\omega=j^*\omega=0$ in the second equality.
The proof is finished.
\end{proof}

The next lemma allows to perform perturbations inside a symplectic
subspace, keeping invariant its symplectic orthogonal space.

\begin{lemma}
\label{symplectic}
Let $W\subset V$ a symplectic subspace.
For any $\eps>0$, there exists $\delta>0$, such that if
$A \colon W\rightarrow W$ is a symplectic linear map $\delta$-close to the
identity map $Id|W$, then there exists a symplectic linear map $B$ over
$V$, $\eps$-close to $Id$ such that $B|W=A$ and
$B| W^{\omega}=Id|W^{\omega}$.
\end{lemma}

\begin{proof}
If $W = V$ we are done, so we suppose $\dim \ W=2m < 2d = \dim V$.
Let $\{e_1,\ldots, e_{2m}\}$ be a symplectic basis of $W$ and  let
$\{e_{2m+1}, \ldots, e_{2d}\}$ be a symplectic basis of $W^{\omega}$.
Hence, $\{e_1,\ldots, e_{2d}\}$ is a symplectic basis of $V$.
Let $< \cdot , \cdot>$ be the inner product for which this basis is
orthonormal and thus $\omega(u,v)=<u,Jv>$, where
$$
J=\left[\begin{array}{cc}
J_W& 0 \\ 0 & J_{W^{\omega}}
\end{array}\right],
\text{ for }
J_{\iota}=\left[\begin{array}{cc}
0& Id_\iota \\ -Id_\iota & 0
\end{array}\right], \quad \iota= W, \ W^{\omega},
$$
being $Id_\iota$ the identity matrix of order $m\times m$ if $\iota=W$,
or the $d-m\times d-m$ identity matrix if $\iota = W^{\omega}$.

Since $A$ is a symplectic linear map over $W$, we have $A^* J_W A=J_W$.
Hence, defining
$$
B=\left[\begin{array}{cc}
A& 0 \\ 0 & Id_{W^{\omega}}
\end{array}\right],
$$
we have that $B^* J B=J$, which implies that $B$ is a symplectic linear
map on $V$, where $B|W=A$ and $B|W^{\omega}=Id_{W^\omega}$.

Therefore, for any $\eps>0$ we can choose $\delta>0$ small enough
depending on the symplectic basis fixed at the beginning, such that if $A$
is $\delta$-close to $Id_W$, then the linear map $B$ is $\eps$-close to
$Id$.
The proof is finished.
\end{proof}

Now, we recall a symplectic version of the well-known Franks Lemma in
\cite{F}, which enable us to perform non-linear perturbations along a
finite invariant set, namely a peridic orbit, from linear perturbations
(in particular, from those given in previous lemmas).

\begin{lemma}[Symplectic Franks Lemma]
\label{franks}
Let $f\in \dw$ and let $\mathcal{U}\subset\dw$ be any neighborhood of $f$.
Then, there exists $\delta>0$ and $\mathcal{U}'\subset \mathcal{U}$ a
small neighborhood of $f$ such that given $g\in \mathcal{U}'$, a finite
$g$-invariant set $\{x_1,\ldots, x_N\}$, a neighborhood $U$ of
$\{x_1,\ldots, x_N\}$ and  symplectic linear maps
$A_i \colon T_{x_i}M \rightarrow T_{g(x_i)}M$ such that
$\|A_i-Dg(x_i)\|\leq \delta$ for all $1\leq i\leq N$, then  there is a
symplectic diffeomorphism $\tilde{g}\in \mathcal{U}$ such that
$\tilde{g}(x) =g(x)$ if $x\in \{x_1,\ldots, x_N\}\cup (M \setminus U)$ and
$D\tilde{g}(x_i)=A_i$ for all $1\leq i\leq N$.
\end{lemma}

The proof in \cite{F} can be extended to the symplectic set using
generating functions.

\subsection{Connecting invariant manifolds.}
\label{ss.connecting}
Given a hyperbolic periodic point $p$ of a symplectic diffeomorphism $f$
we define its {\it stable (resp., unstable)  manifold}, $W^s(p,f)$
(resp.,  $W^u(p,f)$), the subset of points in $M$ whose forward (resp.,
backward) orbit by $f^{\tau(p,f)}$ accumulates on $p$.

\begin{remark}
\label{dec.}
If $f$ is a symplectic diffeomorphism over $M$ and $p$ is a hyperbolic
periodic point of $f$, the stable and unstable manifolds of $p$,
$W^{s}(p,f)$ and $W^u(p,f)$, are Lagrangian submanifolds of $M$, that is
$T_xW^{\iota}(p,f)$ is a Lagrangian subspace for every
$x\in W^{\iota}(p,f)$, $\iota=s$ or $u$.
In particular, $E^s(p,f)$ and $E^u(p,f)$ are Lagrangian subspaces of
$T_pM$.
\end{remark}

\begin{remark}
\label{decomp.}
If $f$ is a partially hyperbolic symplectic diffeomorphism over $M$,
$TM=E^{ss}\oplus E^c\oplus E^{uu}$, then we also recall that  $E^{ss}$ and
$E^{uu}$ are isotropic subbundles of $TM$.
Furthermore, any partially hyperbolic splitting for a symplectic
diffeomorphism is such that $E^c$ is a symplectic subbundle of $TM$
satisfying $(E^c(x))^{\omega}=E^{ss}(x)\oplus E^{uu}(x)$, for every $x\in M$.
In particular, $E^{ss}(x)\oplus E^{uu}(x)$ is a symplectic subspace of $T_xM$.
\end{remark}

Now, given a periodic point  $p$ of $f\in\dw$, if there exists a partially
hyperbolic splitting $T_pM=E^{ss}_k\oplus E^c_k\oplus E^{uu}_k$ for
$Df^{\tau(p,f)}$, with $\dim(E^{ss}_k)=\dim (E^{uu}_k)=k$ and
$\dim E^c_k=2d-2k$, then by \cite{HPS}, there exists a $f^{\tau(p,f)}$
(resp., $f^{-\tau(p,f)}$) {\it local invariant $k$-dimensional strong
stable (resp., unstable) manifold} $W_{k, loc}^{ss} (p)$ (resp.,
$W^{uu}_{k,loc}(p)$ ) tangent to $E^{ss}_k\ ( \text{resp., } E^{uu}_k)$ at
$p$ varying $C^1$-continuously  with respect to the diffeomorphism.
Hence, we define the {\it $k$-dimensional strong stable (resp., unstable)}
manifold of $p$ by
$$
W_{k}^{ss} (p)=\bigcup_{n\in \N} f^{-n}(W_{k, loc}^{ss} (p))
\big(\text{resp., } W_{k}^{uu} (p)=\bigcup_{n\in \N}
f^{n}(W_{k, loc}^{uu} (p))\big).
$$

We say that a periodic point $p$ of $f\in\dw$ is {\it diagonalizable} if
$Df^{\tau(p,f)}(p)$ has only real positive eigenvalues with multiplicity
one. If $p$ is a diagonalizable periodic point then for each $0 < k \le d$
the partially hyperbolic splitting
$T_pM=E^{ss}_k\oplus E^c_k\oplus E^{uu}_k$ is defined.
So, for those points $k$-strong invariant manifolds are defined for any
$0 < k \le d$.
Hence, for any diffeomorphism $f\in \dw$ and any diagonalizable hyperbolic
periodic point $p$ of $f$, denoting by
$\lambda_{1,p} < \ldots < \lambda_{2d,p}$ the distinct simple eigenvalues
of $Df^{\tau(p,f)}(p)$ and by $E_{\lambda_{1,p}} \prec \ldots \prec
E_{\lambda_{2d,p}}$ the respective eigenspaces, we set the
$k$-\emph{dimensional strong stable} (resp., \emph{strong unstable})
\emph{subspace} in $T_pM$, $0<k \le d$ as follows
$$
E^{ss}_k(p)=\bigoplus_{1\leq j\leq k} E_{\lambda_{j,p}} \quad
\left(\text{resp., }
E^{uu}_k(p) = \bigoplus_{2d-k+1\leq j\leq 2d} E_{\lambda_{j,p}}
\right).
$$
For analogy, in this case we denote
$$
E^{c}_k(p)=\bigoplus_{k+1\leq j\leq 2d-k} E_{\lambda_{j,p}}.
$$
It is worth to point out that according to the previous definition
$W_{d}^{ss} (p)=W^{s} (p)$ and $W_{d}^{uu} (p)=W^{u} (p)$.

In this work, by {\it $k$-strong homoclinic intersections}, $k < d$ we
mean non-trivial intersections between $W_{k}^{ss} (p)$ and
$W_{k}^{uu} (p)$.
Moreover, if $q$ is a $k$-strong homoclinic intersection such that
$T_qW_{k}^{ss} (p)\cap T_qW_{k}^{uu} (p)=\{0\}$, then we say $q$ is a
{\it $k$-strong quasi-transversal homoclinic intersection}.
To create such intersections we use a symplectic version of Hayashi
connecting lemma \cite{H}, due to Xia and Wen \cite{XW}.

\begin{proposition}[Theorem F in \cite{XW}]
\label{xw}
Let $z\in M$ be a non-periodic point of $f\in \dw$.
For any $C^1$-neighborhood $\mathcal{U}$ of $f$, there are $\rho>1$,
$L\in \N$ and $\delta_0>0$ such that for any $0<\delta<\delta_0$, and for
any point $x$ outside the tube $\Delta=\cup_{n=1}^L f^{-n} B(z, \delta)$
and any point $y\in B(z,\delta /\rho)$, if the forward $f$-orbit of $x$
intersects $B(z, \delta/\rho)$, then there is a symplectic diffeomorphism
$g\in \mathcal{U}$ such that $g$ = $f$ off $\Delta$ and $y$ is on the
forward g-orbit of $x$.
\end{proposition}

We emphasize that the perturbation $g$ of $f$ in above proposition is a
local perturbation.
That is, $g$ should be different of $f$ only in $\Delta$.

\begin{remark}
\label{r.xw}
Symmetrically, we can restate the previous proposition for a tube along
the positive orbit of $z$, and require that the backward $f$-orbit of $x$
intersects $B(z, \delta / \rho)$, obtaining now that $y$ belongs to the
backward $g$-orbit of $x$.
\end{remark}

Next result is a consequence of Proposition~\ref{xw} which permit us
create $k$-strong homoclinic intersections.

\begin{lemma}
\label{conexao}
Let $f\in \dw$ and let $p,q$ be either hyperbolic or $m$-elliptic
periodic points of $f$ booth having $k$-dimensional strong stable and
unstable manifolds.
For any neighborhood $\mathcal{U}$ of $f$ there exists a diffeomorphism
$g\in \mathcal{U}$ with a $k$-strong heteroclinic intersection for the
analytic continuation $p(g)$ and $q(g)$ of $p$ and $q$, respectively, for
$g$.
That is, there exists an intersection between $W_{k}^{ss} (p(g))$ and
$W_{k}^{uu} (q(g))$.
\end{lemma}

The proof of this lemma uses the fact that transitive diffeomorphisms are
$C^1$-dense in $\dw$.
This is the content of \cite[Theorem 1]{ArBC}, let us state it for
completeness.

We define the \emph{homoclinic class} of a hyperbolic periodic point,
$H(p,f)$, as the closure of transversal intersections between the stable
and unstable manifolds of all points in the orbit of $p$\;:
$H(p,f)=\overline{W^s(orb(p))\pitchfork W^u(orb(p))}$.
It is well-known that homoclinic class is a transitive set and coincides
with the closure of the hyperbolic periodic points homoclinically related
to $p$ (we say that a hyperbolic periodic point $q$ is
\emph{homoclinically related} to $p$ if $W^s(p)\pitchfork W^u(q)\neq
\emptyset$ and $W^u(p)\pitchfork W^s(q)\neq \emptyset$).

\begin{proposition}[Theorem 1 in \cite{ArBC}]
\label{abc}
There exists a residual subset $\mathcal{R}$ of $\dw$ such that if
$f\in \mathcal{R}$ then there exists a hyperbolic periodic point $p$ of
$f$ such that $M=H(p,f)$.
In particular, $f$ is transitive.
\end{proposition}

\begin{proof}[Proof of Lemma~\ref{conexao}]
By Proposition~\ref{abc}, after a perturbation, we can suppose that $f$ is
transitive.
If $W_k^{ss}(p) \cap W_k^{uu}(q) \neq \emptyset$ we are done.
Otherwise, take $z^s\in W_k^{ss}(p)$, $z^u\in W_k^{uu}(q)$, and let
$\mathcal{U}$, $\rho >1$, $L \in \mathbb{N}$, and $\delta_0 >0$ satisfying
simultaneously Proposition~\ref{xw} for $z=z^u$ and Remark~\ref{r.xw} for
$z=z^s$.

We write $\Delta^s=\cup_{n=1}^L f^{-n} B(z^s, \delta)$ and
$\Delta^u=\cup_{n=1}^L f^n B(z^u, \delta)$.
Since $L$ is finite and $W_k^{ss}(p) \cap W_k^{uu}(q) = \emptyset$, we can
choose $\delta>0$ small enough such that $\Delta^s \cap \Delta^u =
\emptyset$ and
$$
f^n(z^s)\notin \Delta^s\cup \Delta^u \text{ and }
f^{-n}(z^u)\notin \Delta^s\cup \Delta^u, \text{ for every } n\in \N.
$$

Since $f$ is transitive we can find $x\in B(z^s,\delta/\rho)$ such that
$f(x)\notin \Delta^s\cup \Delta^u$ and $f^m(x)\in B(z^u, \delta / \rho)$
for a positive integer $m$.
Now, by choice of $\delta$ and $x$, applying Proposition~\ref{xw}
simultaneously for $z^s$ and $z^u$ (which is possible since the
perturbation is a local perturbation of $f$) we can find a symplectic
diffeomorphism $g$, $C^1$-close to $f$, such that $f(x)= g^{-n+1}(z^u)$
and $g^{-1}(f(x))=z^s$.
Therefore, $z^s$ belongs to the backward orbit of $z^u$, and since
$g=f$ outside $\Delta^s\cup \Delta^u$ we have
$z^s, z^u\in W_k^{ss}(p(g))\cap W_k^{uu}(p(q))$.
The lemma is proved.
\end{proof}

\subsection{Periodic symplectic linear systems}
\label{ss.linear-systems}
We recall the concept of periodic linear systems with transitions in the
symplectic scenario as done in \cite{HT}.
For the general definition and more details see \cite{BDP}.

Let $f$ be a homeomorphism defined on a topological space $\Sigma$.
Let $\mathcal{E}$ be a locally trivial vector bundle over $\Sigma$ such
that for every $x \in \Sigma$, $\mathcal{E}(x)$ is a symplectic vector
space of same dimension and endowed with the same symplectic form
$\omega$.
We define $\mathcal{S}(\Sigma, f, \mathcal{E})$ the set of maps
$A \colon \mathcal{E} \to \mathcal{E}$ such that for every $x \in \Sigma$
the induced map $A(x, \cdot)$ is a linear symplectic isomorphism
$\mathcal{E}(x) \to \mathcal{E} (f(x))$, that is, $\omega(u,v) =
\omega(A(u),A(v))$.
Thus, $A(x, \cdot)$ belongs to $\mathcal{L}_\omega(\mathcal{E}(x),
\mathcal{E}(f(x)))$.
We define a norm $| \cdot |$ on $\mathcal{L}_\omega(\mathcal{E}(x),
\mathcal{E}(f(x)))$ induced by Euclidean metrics of $\mathcal{E}(x)$ and
$\mathcal{E}(f(x))$\;:
$$
|A(x,\cdot)| = \sup \{ |A(x,v)|, v \in \mathcal{E}(x), |v| = 1 \}.
$$
For $A \in \mathcal{S}(\Sigma, f, \mathcal{E})$ we set $|A| =
\sup \{ |A(x, \cdot)| \colon x\in \Sigma \}$.
Then we define the \emph{norm} of $A \in \mathcal{S}(\Sigma, f,
\mathcal{E})$ as $\| A \| = \max \{|A|,|A^{-1}| \}$.

A {\it linear symplectic system (or linear symplectic cocycle over $f$)}
is a 4-tuple $(\Sigma, f, \mathcal{E},A)$, where $\Sigma$ is a topological
space, $f$ is a homeomorphism on $\Sigma$, $\mathcal{E}$ is a symplectic
vector bundle defined over $\Sigma$, and $A \in \mathcal{S}(\Sigma, f,
\mathcal{E})$ with $\|A\| < \infty$.
When all points in $\Sigma$ are periodic points of $f$ we say that
$(\Sigma, f, \mathcal{E},A)$ is \emph{periodic}.

Now, we recall the concept of \emph{linear systems with transitions}.
Given a set $\mathcal{B}$, a {\em word} with letters in $\mathcal{B}$ is a
finite sequence of elements of $\mathcal{B}$.
The product of the word $[a]= (a_1,\dots ,a_n)$ by $[b]= (b_1,\dots ,b_m)$
is the word $(a_1, \dots , a_n, b_1, \dots, b_m)$.
We say a word is \emph{not a power} if $[a] \neq [b]^k$ for every word
$[b]$ and $k > 1$.

With this notation, for a periodic symplectic linear system
$(\Sigma, f, \mathcal{E}, A)$ if we consider the word
$[M_A (x)] = (A(f^{n-1}(x)), \dots , A(x)),$
where $n$ is the period of $x\in \Sigma$, then
the matrix $M_A(x)$ is the product of the letters of the word $[M_A (x)]$,
that is,
$$
M_A (x) = A(f^{n-1}(x)) A(f^{n-2}(x)) \ \dots \   A(x).
$$
A periodic linear system is {\it diagonalizable} at the point $x\in\Sigma$
if  $M_A(x)$ has only real eigenvalues of multiplicity one.

\begin{definition}[Definition 1.6 of \cite{BDP}]
\label{d.transition}
Given $\varepsilon > 0$, a periodic symplectic linear system
$(\Sigma, f, \mathcal{E},A)$ {\em admits $\varepsilon$-transitions} if for
every finite family of points $x_1, \dots , x_n = x_1 \in \Sigma$ there is
an orthonormal system of coordinates of the linear bundle $\mathcal{E}$ so
that $(\Sigma, f, \mathcal{E},A)$ can now be considered as a system of
matrices $(\Sigma, f, A)$, and for any $(i, j) \in \{1,\dots ,n\}^2$ there
exist $k(i, j) \in \mathbb{N}$ and a finite word
$[t^{i,j}]= (t_1^{i,j}, \dots , t_{k(i,j)}^{i,j})$ of symplectic matrices,
satisfying the following properties:
\begin{enumerate}
\item For every $m \in \mathbb{N}$, $\imath = (i_1, \dots, i_m) \in
\{1,\dots,n\}^m$, and $\alpha=(\alpha_1,\dots ,\alpha_m) \in \mathbb{N}^m$
consider the word
\begin{align*}
[W(\imath, \alpha)] & = [t^{i_1,i_m}][M_A(x_{i_m})]^{\alpha_m}
[t^{i_m,i_{m-1}}][M_A(x_{i_{m-1}})]^{\alpha_{m-1}} \dots
\\ & \quad \dots
[t^{i_2,i_1}][M_A(x_{i_1})]^{\alpha_1},
\end{align*}
where the word
$[W(\imath, \alpha)]$ is not a power.
Then there is $x(\imath, \alpha) \in \Sigma$ such that
\begin{itemize}
\item the length of $[W(\imath, \alpha)]$ is the period of
$x(\imath,\alpha)$;
\item the word $[M_A (x(\imath, \alpha))]$ is $\varepsilon$-close to
$[W(\imath, \alpha)]$ and there is an $\varepsilon$-sym\-plec\-tic perturbation
$\tilde A$ of $A$ such that the word $[M_{\tilde{A}} (x(\imath,\alpha))]$
is $[W(\imath, \alpha)]$.
\end{itemize}

\item One can choose $x(\imath, \alpha)$ such that the distance between
the orbit of $x(\imath, \alpha)$ and any point $x_{i_k}$ is bounded by some
function of $\alpha_k$ which tends to zero as $\alpha_k$ goes to infinity.
\end{enumerate}
\end{definition}

Given $\imath ,\alpha$ as above, the word $[t^{i, j}]$ is an
$\varepsilon$-transition from $x_j$ to $x_i$.
We call {\em $\varepsilon$-transition matrices} the matrices $T_{i,j}$
which are the product of the letters composing $[t^{i, j}]$.
We say a periodic linear system {\em admits transitions} if
for any $\varepsilon > 0$ it admits $\varepsilon$-transitions.

\begin{remark}
Let $x_1, \dots , x_n=x_1$ be in $\Sigma$ and let $[t^{i,j}]$ be an
$\eps$-transition from $x_j$ to $x_i$.
Then for every $\alpha, \beta \ge 0$ the word
$$
([M_A(x_i)]^{\alpha} [t^{i,j}] [M_A(x_j)]^{\beta})
$$
is also an $\eps$-transition from $x_j$ to $x_i$.
Further, if $[t^{j,k}]$ is an $\eps$-transition from $x_k$ to $x_j$, then
the word $[t^{i,j}][t^{j,k}]$ is an $\eps$-transition from $x_k$ to $x_i$.
In particular, for any $\eps>0$ and $x \in \Sigma$ we can consider non
trivial $\eps$-transitions from $x$ to itself.
\label{rmktransition}\end{remark}

The following lemma gives an example of linear systems having symplectic
transitions.
It is a symplectic version of [Lemma 1.9 in \cite{BDP}].

\begin{lemma}[Lemma 4.5 in \cite{HT}]
\label{ex}
Let $f$ be a symplectic diffeomorphism and let $p$ be a hyperbolic
periodic point of $f$.
The derivative $Df$ induces a continuous periodic symplectic linear system
with transitions on the set $\Sigma$ formed by hyperbolic periodic points
homoclinically related  to $p$.
\end{lemma}

A nice property of periodic linear systems $(\Sigma, f, \mathcal{E},A)$
admitting transitions is the existence of arbitrarily small perturbation
of $A$ which is diagonalizable and defined on a dense subset of $\Sigma$,
see \cite[Lemma 4.16]{BDP} (and \cite[Lemma 4.7]{HT} for a symplectic
version).

%==========================================================================
\section{Proof of Theorem~\ref{t.trichotomy}}
\label{s.proof_t.trichotomy}
We start this section proving that after a small pertubation we obtain a
diffeomorphism in $\ph^1(m)$ having a \emph{nice} non-hyperbolic periodic
point.
In the sequel, we use this result to prove Theorem \ref{t.trichotomy}

\begin{proposition}
\label{p.identity}
Let $f\in \inte(\mathcal{PH}_{\omega}^1(m))$.
For any small neighborhood $\mathcal{U}\subset \ph^1(m)$ of $f$ and
$\eps>0$ there exists a diffeomorphism $g\in \mathcal{U}$ and a periodic
point $p$ of $g$ such that $Dg^{\tau(p,g)}(p)|E_{m}^c=Id$.
Moreover the orbit of $p$ is $\eps$-dense in $M$.
\end{proposition}

The proof of this proposition is a direct consequence of the next result
and Proposition~\ref{abc}.
Let us mention that part of this proposition is given in
\cite[Theorem 3.5]{ABW}.

\begin{proposition}[Proposition 5.3 in \cite{HT}]
\label{HT}
For any $\eps>0$, and $K>0$ there is $l>0$ such that any symplectic
periodic $2d$-dimensional linear system $(\Sigma,f, \mathcal{E}, A)$
bounded by $K$ (i.e. $\|A\|<K$) and having symplectic transitions
satisfies the following,
\begin{itemize}
\item  either $A$ admits an $l$-dominated splitting,
\item or there are a symplectic $\eps$-perturbation $\tilde{A}$ of $A$ and
a point $x\in \Sigma$ such that $M_{\tilde{A}}(x)$ is the identity matrix.
\end{itemize}
\end{proposition}

\begin{remark}
We remark that the periodic point $x$ in the second item of the
previous proposition can be find with $\eps$-dense orbit in $\Sigma$, for
any $\eps>0$.
\label{rmkHT}
\end{remark}

\begin{proof}[Proof of Proposition \ref{p.identity}]
Let $TM=E^s\oplus E^c\oplus E^u$ be the partially hyperbolic splitting
with unbreakable center given by $f$ over $M$.
Since $f\in \inte(\mathcal{PH}_{\omega}^1(m))$, by Proposition \ref{abc}
we can suppose, after a perturbation, that $M=H(p,f)$.
We denote by $\Sigma$ the set of hyperbolic periodic points homoclinically
related to $p$, which are dense in $M$, and consider thus the periodic
symplectic linear system $(\Sigma,f, E^c, Df|E^c)$ having symplectic
transitions.

Provided that $f\in \mathcal{PH}^1_{\omega}(m)$, the vector bundle $E^c$
admits no dominated splitting for $Df$.
It follows from Proposition~\ref{HT} and Remark~\ref{rmkHT} that there
exists $x\in \Sigma$ with $\eps$-dense orbit in $M$ and a symplectic
perturbation $\tilde{A}$ of $Df|E^c$ along the orbit of $x$, such that
$M_{\tilde{A}}(x) = Id$.

Hence, let $\delta>0$ be a small constant and $\mathcal{U}$ a neighborhood
of $f$ given by Franks Lemma (Lemma~\ref{franks}), we use
Lemma \ref{symplectic} to find symplectic linear maps $A_i$,
$\delta$-close to $Df(f^i(x))$, $0\leq i<\tau(x)$, such that
$A_i|E^c=\tilde{A}(f^i(x))$ and
$A_i|(E^{ss}\oplus E^{uu})=Df|(E^{ss}\oplus E^{uu})$.
Thus, there exists $g\in \mathcal{U}$ such that $x$ still is a periodic
point of $g$ and $Dg(g^{i}(x))=A_i$, for any $0\leq i <\tau(x)$, which
implies
$$
Dg^{\tau(x)}|E^c(x)=M_{\tilde{A}}(x)=Id|E^c(x).
$$
Proving the proposition.
\end{proof}

Using Proposition \ref{p.identity} we are able to prove
Theorem~\ref{t.trichotomy}.

\begin{proof}[Proof of Theorem \ref{t.trichotomy}]
By Dolgopyat and Wilkinson \cite{DW}, for any $1\leq m<d$, there exists
an open and dense subset $\displaystyle \widetilde{\mathcal{PH}}^1_{\omega}(m)
\subset \inte(\mathcal{PH}^1_{\omega}(m))$ such that every
$f \in \displaystyle\widetilde{\mathcal{PH}}^1_{\omega}(m)$ is robustly
transitive partially hyperbolic symplectic diffeomorphisms.

Recall we are denoting the set of Anosov symplectic $C^1$-diffemorphisms
by $\mathcal{A}$.
We write
$$
\widetilde{\mathcal{PH}}^1_{\omega}(d)=\dw\setminus \overline{\mathcal{A}
\cup \bigcup_{1\leq m<d}\widetilde{\mathcal{PH}}^1_{\omega}(m)}.
$$
Note that $\widetilde{\mathcal{PH}}^1_{\omega}(d)$ coincides with
$\mathcal{PH}^1_{\omega}(d)$ (the complement of the closure of partially
hyperbolic and Anosov diffeomorphisms).
In fact, if $f\in \dw$ is partially hyperbolic (or Anosov), then after a
perturbation we can assume that the center bundle $E^c$ of $f$ has no
dominated splitting, say $\dim E^c=2m$, $0 \le m < d$ in a neighborhood
of $f$.
Hence, by continuity of the partially hyperbolic splitting,
$f\in \inte(\mathcal{PH}^1_{\omega}(m))$.

Given $1\leq m\leq d$ and $n\in \N$ we denote by
$\mathcal{B}_{n,m} \subset \widetilde{\mathcal{PH}}^1_{\omega}(m)$ the
subset of diffeomorphisms  $g$ having a $m$-elliptic periodic point, with
$1/n$-dense orbit in $M$.
Since $m$-elliptic periodic points are robust, $\mathcal{B}_{n,m}$ is an
open set.

Let $f\in \widetilde{\mathcal{PH}}^1_{\omega}(m)$ and $n\in \N$.
By Proposition \ref{p.identity}, there exists a diffeomorphism $g \in
\inte \mathcal{PH}^1_{\omega}(m)$, $C^1$-close to $f$, having a periodic
point $p$ with $1/n$-dense orbit in $M$, such that
$Dg^{\tau(p,g)}|E^c(p)=Id$.
Let $\{e_1, \ldots, e_{2m}\}$ be a symplectic basis in $E^c(p)$,
where the subspace spanned by $\{e_i, e_{m+i}\}$, $1 \le i \le m$, is
a symplectic subspace.
We rearrange the vectors and we consider the basis
$B=\{e_1, e_{m+1},\ldots, e_{i}, e_{m+i}, \ldots, e_{m}, e_{2m}\}$
of $E^c(p)$.
Thus for any small positive values $\alpha, \beta>0$ we can define a
symplectic linear map in $T_pM$ induced by the following matrix with
respect to basis $B$:
$$
A=\left[\begin{array}{ccccc}
\tilde{A}& 0 & 0 & \ldots & 0
\\ 0 & \tilde{A}& 0&\ldots &0
\\
\vdots &&&&
\\0&\ldots&&& \tilde{A}
\end{array}\right],
\text{ where }
\tilde{A}=\left[\begin{array}{cc}
\cos \alpha & \sin \alpha \\ - \sin \alpha & \cos \alpha
\end{array}\right].
$$
Note this symplectic linear map restrict to the symplectic plane generated
by $\{e_i,\ e_{m+i}\}$ is a small rotation, whenever $\alpha$ is small
enough, for any $1\leq i\leq m$.
So, we can suppose $A$ arbitrary close to $Id$.
Hence, by Lemma~\ref{symplectic} and Remark~\ref{decomp.}, we can find a
symplectic linear map $B \colon T_pM \to T_pM$ arbitrary close to $Id$,
such that $B|E^c=A$ and $B|(E^{ss}\oplus E^{uu})=Id|(E^{ss}\oplus E^{uu})$.
Taking $C=B \circ Dg(g^{\tau(p)-1})$ which is a symplectic linear map
close to $Dg(g^{\tau(p)-1})$, we can use Franks Lemma to perform a local
perturbation of $g$ and find a diffeomorphism $h$ $C^1-$close to $g$, such
that $p$ still is a periodic point of $h$ and
$Dh^{\tau(p,h)}(p)=B\circ Dg^{\tau(p,g)}(p)$.
Then $Dh^{\tau(p,h)}|E^c(p)=A$, which implies $p$ is an $m$-elliptic
periodic point.
Since this perturbation keeps the orbit of $p$, this $m$-elliptic periodic
point still have ($1/n$)-dense orbit in $M$, which implies
$h\in \mathcal{B}_{n,m}$.
Note, when $m=d$, $p$ is a totally elliptic periodic point.

Therefore the sets $\mathcal{B}_{n,m}$ are open and dense inside
$\widetilde{\mathcal{PH}}^1_{\omega}(m)$, which implies
$$
\mathcal{R}=\mathcal{A}\cup \left(\bigcup_{1\leq m\leq d} \bigcap_{n\in \N}
\mathcal{B}_{n,m}\right)
$$
is a residual subset of $\dw$.

To finish, we remark that diffeomorphisms in $\mathcal{R}$ satisfies one,
and only one, of the three items in  Theorem \ref{t.trichotomy}.
\end{proof}

%=========================================================================

\section{Bounds for entropy: proof of Theorem \ref{t.entropy}}
\label{s.bounds-entropy}

Using periodic symplectic linear systems with transitions we show that
the supremum in
$$
S_m(f)=\sup\left\{\frac{1}{\tau(p,f)}\log
\sigma(Df^{\tau(p,f)}|E^c(p)) \colon p \text{ periodic hyperbolic
point of } f \right\}
$$
is achieved taking account just diagonalizable
periodic points.

\begin{remark}
\label{r.semicontinuity}
Note that $S_m(\cdot)$, $0 < m \le d$, is a lower semicontinuous map.
Indeed, let $\per_h^n (f)$ be the set of hyperbolic periodic points of
period smaller or equal to $n$.
Provided that hyperbolic periodic points are robust, the function
$S_m^n(\cdot) := \sup\left\{\frac{1}{\tau(p,f)}
\log\sigma(Df^{\tau(p,f)}|E^c(p)),\; p\in \per_h^n(f) \right\}$ are
continuous function, and then
$S_m(\cdot)$ is lower semicontinuous.
\end{remark}

We denote the set of hyperbolic periodic point of $f$ by $\per_h(f)$.

\begin{proposition}
\label{p.diagonalizable}
There exists a residual subset $\mathcal{R}_m \subset \inte \ph^1(m)$,
$0<m\leq d$, such that if $f\in \mathcal{R}_m$ then
$$
S_m(f)=\sup\left\{\frac{1}{\tau(p,f)}\log\sigma(Df^{\tau(p,f)}|E^c(p));\,\;
p\in \per_h(f) \text{ is diagonalizable}\right\}.
$$
\end{proposition}

A key point in the proof of this proposition is the next technical result
that allows perturbations of symplectic linear systems to get
diagonalizable systems with close largest absolute eigenvalues.

\begin{lemma}
\label{exp}
Let $(\Sigma, f, \mathcal{E},A)$ be a periodic symplectic linear system
with transition.
For any $\eps>0$, and $x\in \Sigma$ there exists $y\in \Sigma$ and an
arbitrarily small symplectic perturbation $\tilde{A}$ of $A$ defined
on the orbit of $y$, such that $M_{\tilde{A}}(y)$ is diagonalizable.
Moreover if $\lambda_x$ (resp. $\lambda_y$) denotes the eigenvalue of
$M_A(x)$ (resp. $M_{\tilde{A}}(y)$) with largest absolute value, then
$$
\left|\frac{1}{\tau(x)}\log |\lambda_x|-
\frac{1}{\tau(y)}\log |\lambda_y| \right|<\eps,
$$
where $\tau(x)$ (resp. $\tau(y)$) denotes the period of $x$ (resp. $y$).
\end{lemma}

\begin{proof}
After an arbitrarily small symplectic perturbation of $A$ along a
periodic orbit of $x$, we can assume that $M_A(x)$ has only simple
eigenvalues and that any complex eigenvalue has rational argument.
Hence, supposing that $\mathcal{E}$ is a $2d$-dimensional vector bundle,
we consider the partially hyperbolic splitting
$\R^{2d}=F_1\oplus\ldots \oplus F_n$ given by the eigenspaces associated to
the eigenvalues of $M_A(x)$, which implies $\dim\ F_i = 1, 2$, and as a
consequence of the symplectic structure we have for every distinct
$1\leq i, j\leq n$:
\begin{itemize}
\item $F_i$ is an isotropic subspace,
\item $F_i\oplus F_j$ is a symplectic subspace if $i+j=n+1$, and
\item $(F_i\oplus F_j)^{\omega}=\bigoplus_{k\neq i,j} F_k$, if $i+j=n+1$.
\end{itemize}
The fact that all eigenvalues of $M_A(x)$ has rational argument implies
that there exists a positive integer $k$ such that $(M_A(x))^k$ has only
real eigenvalues.
However, if $\dim\ F_i=2$ then $(M_A(x))^k|F_i$ has a real eigenvalue with
multiplicity two.
We use Lemma \ref{symplectic} to find a symplectic linear map $H_i$
arbitrary close to identity, such that $H_i|F_j=Id$, if $j\neq i, n+1-i$,
and $H_i (M_A(x))^k|F_i\oplus F_{n+1-i}$ have four distinct real
eigenvalues.
Moreover, for any $\eps>0$ we can choose such $H_i$ such that defining
$M_{1,i}=H_i (M_A(x))^k$, if $\xi$ is an eigenvalue of
$M_{1,i}|F_i\oplus F_{n+1-i}$ then there is an  eigenvalue $\lambda$ of
$M_A(x)|F_i\oplus F_{n+1-i}$ such that:
\begin{equation}
\left|\log |\lambda| -  \frac{1}{k}\log |\xi|\right|<\frac{\eps}{2}.
\label{eigen}
\end{equation}
Hence, given $\eps>0$, we can use the existence of the above linear maps
$H_i$ defined on $F_i\oplus F_{n+1-i}$, when $\dim E_i=2$, to find a
symplectic linear map $H$ arbitrarily close to $Id$ such that
$M_1=H(M_A(x))^k$ has only real eigenvalues with multiplicity one, and
\eqref{eigen} holds for every $1\leq i\leq n$.
We denote by $E_i$ the one-dimensional $M_1$-invariant eigenspaces,
$1\le i \le 2d$.

Since the linear system has transitions, there exists a non trivial word
$[t]=(t_1,\ldots, t_{r})$ of symplectic matrices which is a
$(\eps/2)$-transition from $x$ to itself, see Remark \ref{rmktransition}.
We denote by $T$ the symplectic matrix obtained by the product of the
matrices in the word $[t]$.
After an arbitrarily small symplectic perturbation of the matrix $t_{1}$,
if necessary, we can suppose that
$$
T(E_{2d})\cap (E_1\oplus\ldots\oplus E_{2d-1}) = \{0\}\text{ and }
T^{-1}(E_1)\cap (E_2\oplus\ldots\oplus E_{2d})=\{0\}.
$$
Thus by the choice of the partially hyperbolic splitting on
$\R^{2d}$, $(M_1)^jT(E_{2d})$ converges to $E_{2d}$ when $j$ goes to
infinity.
Hence, taking $j_{2d}$ large enough, by Lemma~\ref{isotropic2} we can find
a symplectic linear map $L_{2d}$ close to $Id$, such that
$L_{2d}(M_1)^{j_{2d}}T(E_{2d})=E_{2d}$ and
$L_{2d}|(E_1\oplus \ldots \oplus E_{2d-1})=Id$.

Analogously, provided that $(M_1)^{-j}T^{-1}(E_1)$ converges to $E_1$ when
$j$ goes to infinity, we can choose $j_1>0$ to find a symplectic linear
map $L_1$ arbitrarily close to $Id$, such that $L_1(E_1)=(M_1)^{-j_1}T^{-1}(E_1)$ and
$L_1|(E_2\oplus\ldots\oplus E_{2d})=Id$.
Therefore, defining
$$
\tilde{M}_1=L_{2d}(M_1)^{j_{2d}}T(M_1)^{j_1}L_1$$
we have that $\tilde{M}_1(E_1)=E_1$ and $\tilde{M}_1(E_{2d})=E_{2d}$.
Once again, as a consequence of the symplectic structure, $\tilde{M}_1$
also satisfies
$$
\tilde{M}_1(E_2\oplus\ldots\oplus E_{2d-1})=E_2\oplus\ldots\oplus E_{2d-1}.
$$
In fact, if this is not true, then there exist $v\in E_2 \oplus \ldots
\oplus E_{2d-1}$ and $u\in E_1\oplus E_{2d}$ such that
$\omega(\tilde{M}_1(v), u)\neq 0$.
On the other hand, by construction,
$\tilde{M}_1^{-1}(u)\in E_1\oplus E_{2d}$ and then
$\omega(v, \tilde{M}_1^{-1}(u))=0$, which gives a contradiction since
$\tilde{M}_1$ is symplectic.

Proceeding as before, we can find $j_2, j_{2d-1}$ positive integers
sufficiently large, symplectic linear maps $L_2$ and $L_{2d-1}$
arbitrarily close to $Id$, such that
$L_2(E_2)=({M}_1)^{-j_2}(\tilde{M}_1)^{-1}(E_2)$,
$L_2|(E_1 \oplus E_3\oplus\ldots\oplus E_{2d})=Id$,
$L_{2d-1}(M_1)^{j_{2d-1}}\tilde{M}_1(E_{2d-1})=E_{2d-1}$, and
$L_{2d-1}|(E_1\oplus \ldots \oplus E_{2d-2} \oplus E_{2d})=Id$.
So, defining
$$
\tilde{M}_2=L_{2d-1}(M_1)^{j_{2d-1}}\tilde{M}_1(M_1)^{j_2}L_2
$$
we have that $\tilde{M}_2 (E_i) = E_i$, for $i=1,2,2d-1, 2d$, and
$$
\tilde{M}_2(E_3\oplus\ldots\oplus E_{2d-2})=E_3\oplus\ldots\oplus E_{2d-2}.
$$
Repeating the above process finitely many times we also can find
symplectic linear maps $L_i$ close to identity, for any
$3\leq i\leq 2d-2$, such that the maps
$$
\tilde{M}_k=L_{2d-k+1}(M_1)^{j_{2d-k+1}}\tilde{M}_{k-1}(M_1)^{j_k}L_k,
$$
are well defined for any $3\leq k\leq d$, satisfying
$\tilde{M}_k (E_i) = E_i$, for $i=1,\ldots, k, 2d-k+1,\ldots, 2d$ and
$\tilde{M}_k(E_{k+1}\oplus\ldots\oplus E_{2d-k})=E_{k+1}\oplus
\ldots\oplus E_{2d-k}$,
where the last equality holds only when $k< d$.
In particular, $\tilde{M}= \tilde{M_d}$ preserves $E_i$ for every
$1\leq i \leq 2d$, i.e., $\tilde{M}(E_i)=E_i$.

Now, since $E_i$ is an one-dimensional subspace, $1\leq i\leq 2d$, we can
choose $l>0$ large enough and define $M=(M_1)^l \tilde{M}$, such that if
$\mu_i$ and $\xi_i$ denote the eigenvalues of $M|E_i$ and $M_1|E_i$,
respectively; and $\tau=\tau(x)k(\jmath+l)+r$, where
$\jmath = j_1, \ldots, j_{2d}$ (recall $\tau(x)$ is the period of
$x\in \Sigma$ and $r$ is the length of $[t]$) then we have
\begin{equation}
\left|\frac{1}{\tau}\log |\mu_i |-\frac{1}{k\tau(x)}\log |\xi_i|\right|<
\frac{\eps}{2}.
\label{eigen3}
\end{equation}
Hence, if we denote by $\mu_M$ (resp. $\lambda_x$) the eigenvalue of
largest absolute value of $M$ (resp. $M_A(x)$),  then equations
\eqref{eigen} and \eqref{eigen3} give
\begin{equation}
\left|\frac{1}{\tau}\log |\mu_M |-\frac{1}{\tau(x)}\log |\lambda_x|\right|
<\eps.
\label{eigen2}
\end{equation}

Thus, since $[t]$ is a non trivial $(\eps/2)$-transition from $x$ to
itself, there exists $y\in \Sigma$ such that $[M_A(y)]$ is
$(\eps/2)$-close to $[\tilde{M}]=[M_A(x)]^{k(l+j_{d+1}+\ldots + j_{2d})}
[t][M_A(x)]^{j_1+\ldots + l_d}$, and moreover $\tau(y)$ is equal to the
length of $[\tilde{M}]$ which is $\tau$.

Now, since $H$ and $L_i$, $1\leq i\leq 2d$, are symplectic linear maps
close to $Id$, then the matrix $M_A(y)$ is close to $M$, which implies
that there exists an arbitrarily small symplectic perturbation $\tilde{A}$ of $A$
defined on the orbit of $y$, such that $M_{\tilde{A}}(y)=M$.
Therefore, $M_{\tilde{A}}(y)$ is diagonalizable and if $\lambda_y$ denotes
the eigenvalue with largest absolute value of $M_{\tilde{A}}(y)$ we have,
by \eqref{eigen2}, that
$$
\left|\frac{1}{\tau(y)}\log |\lambda_y | -
\frac{1}{\tau(x)}\log |\lambda_x|\right|<\eps.
$$
This complete the proof.
\end{proof}

\begin{remark}
In the previous lemma if we have the additional hypothesis that the linear
bundle $\mathcal{E}$ has a dominated splitting for $f$,
$\mathcal{E}=\mathcal{E}_1\oplus \ldots \oplus\mathcal{E}_n$, then the
diagonalizable periodic point $y$ could be found such that
$M_{\tilde{A}}(y)$ keeps invariant the subbundles $\mathcal{E}_i$ and
moreover
$$
\left|\frac{1}{\tau(x)}\log |\lambda_{x,i}|-
\frac{1}{\tau(y)}\log |\lambda_{y,i}|\right|<\eps,
$$
where $\lambda_{x,i}$ and $\lambda_{y,i}$ denote the eigenvalues with
largest absolute value of $M_A(x)|\mathcal{E}_i$ and
$M_{\tilde{A}}(y)|\mathcal{E}_i$, respectively, for every $1\leq i\leq n$.
\label{reexp}
\end{remark}

\begin{proof}[Proof of Proposition~\ref{p.diagonalizable}]
For each $0 < m \le d$ we define
$$
\tilde{S}_m(f) = \sup\left\{\frac{1}{\tau(p,f)}
\log\sigma(Df^{\tau(p,f)}|E^c(p)),\; p\in \per_h(f) \text{ is
diagonalizable}\right\}.
$$
Similar to ${S}_m(f)$ the maps $\tilde{S}_m(f)$ are lower semicontinuous
for each $0 < m \le d$, see Remark~\ref{r.semicontinuity}.
Hence there exists a residual subset $\mathcal{R}_m \subset\inte \ph^1(m)$
where  $\tilde{S}_m(f)$ is continuous.
Taking $f\in \mathcal{R}_m$, for any $\eps>0$ there exists a small
neighborhood $\mathcal{U} \subset\inte \ph^1(m)$ of $f$ such that
\begin{equation}
\tilde{S}_m(g)<\tilde{S}_m(f)+\frac{\eps}{3}, \text{ for every }
g\in \mathcal{U}.
\label{e3}\end{equation}

By definition of $S_m(f)$, there exists a hyperbolic periodic point $p$ of
$f$ such that
\begin{equation}
S_m(f)<
\frac{1}{\tau(p,f)}\log\sigma(Df^{\tau(p,f)}|E^c(p))+\frac{\eps}{3}.
\label{e1}\end{equation}

It follows from Lemma~\ref{ex} that $Df$ induces a periodic symplectic
linear system with transition $(\Sigma, f, TM, Df)$, where $\Sigma$ is the
set of hyperbolic periodic points of $f$ homoclinicaly related to $p$.
We can suppose $\Sigma$ non-trivial since $f$ belongs to a residual
subset, see \cite{X}.

Let $\delta>0$ be a small constant given by Franks Lemma
(Lemma~\ref{franks}) for $f$ and the neighborhood $\mathcal{U}$.
Provided that $TM=E^s\oplus E^c\oplus E^u$, by Lemma~\ref{exp} and
Remark~\ref{reexp} there exist $\tilde{p}\in \Sigma$ and a symplectic
$\delta$-perturbation $\tilde{A}$ of $Df$ along the orbit of $\tilde{p}$
such that $\tilde{p}$ is diagonalizable and moreover
$$
\frac{1}{\tau(p,f)}\sigma(Df^{\tau(p,f)}|E^c(p))<
\frac{1}{\tau(\tilde{p})}\sigma(M_{\tilde{A}}(\tilde{p})|E^c) +
\frac{\eps}{3}.
$$
Hence, using Franks Lemma, we can find a symplectic diffeomorphism
$g\in \mathcal{U}$ such that $\tilde{p}$ is a diagonalizable hyperbolic
periodic point of $g$ satisfying
\begin{equation}
\frac{1}{\tau(p,f)}\sigma(Df^{\tau(p,f)}|E^c(p))<
\frac{1}{\tau(\tilde{p},g)}\sigma(Dg^{\tau(\tilde{p},g)}|E^c(\tilde{p}))+
\frac{\eps}{3}.
\label{e2}\end{equation}
Using respectively (\ref{e1}), (\ref{e2}), definition of $\tilde{S}_m(g)$,
and (\ref{e3}) we obtain
\begin{align*}
S_m(f)&<
\frac{1}{\tau(p,f)}\log\sigma(Df^{\tau(p,f)}|E^c(p))+\frac{\eps}{3} \\
& <\frac{1}{\tau(\tilde{p},g)}\log\sigma(Dg^{\tau(\tilde{p},g)}|
E^c(\tilde{p}))+\frac{2\eps}{3} \\
& \leq \tilde{S}_m(g)+\frac{2\eps}{3} < \tilde{S}_m(f)+\eps.
\end{align*}

Therefore, since $\eps>0$ is arbitrary, we have $S_m(f)\leq
\tilde{S}_m(f)$, for every $f \in \mathcal{R}_m$.
Which finishes the proof since $\tilde{S}_m(f)\leq S_m(f)$ by definition.
\end{proof}

The next proposition is the main technical result in this paper.
\begin{proposition}
\label{p.estimate}
Let $0<m\leq d$ and $f\in \inte(\ph^1(m))$.
If $p$ is a diagonali\-zable hyperbolic periodic point of $f$, then for
any neighborhood $\mathcal{U}$ of $f$ and any large positive integer
$n$, there exists a diffeomorphism $g\in \mathcal{U}$, such that $p$
still is a diagonalizable hyperbolic periodic point of $g$.
Moreover there exists a hyperbolic basic set $\Lambda(p,g)\subset
H(p,g)$ such that
$$
h_{top}(g|\Lambda(p,g))>\frac{1}{\tau(p,g)}
\log\sigma(Df^{\tau(p,g)}|E^c)-\frac{1}{n}.
$$
\end{proposition}

We postpone the proof of this proposition to Section~\ref{section5}.
Now, let us prove Theorem~\ref{t.entropy}.

\medskip

\begin{proof}[Proof of Theorem~\ref{t.entropy}]
For any positive integer $n>0$, and every $0<m\leq d$ we define
$\mathcal{B}_{m,n}\subset \inte(\ph^1(m))$ the subset of diffeomorphisms
$g$ such that
$$
h_{top}(g)>S_m(g)-\frac{1}{n}.
$$

Since $S_m(\cdot)$ is a lower semicontinuous map defined in $\ph^1(m)$,
there is a residual subset $\mathcal{R}_m^*\subset \inte(\ph^1(m))$
where $S_m(\cdot)$ is continuous.
We can choose $\mathcal{R}_m^*$ as a subset of $\mathcal{R}_m$
in Proposition~\ref{p.diagonalizable}.

Let us fix  some $0<m\leq d$.
Given $f\in \mathcal{R}_m^*$ and $n > 0$ we consider a small neighborhood
$\mathcal{U} \subset \inte(\ph^1(m))$ of $f$ such that
\begin{equation}
S_m(f)>S_m(\xi)-\frac{1}{4n}, \text{ for every } \xi\in \mathcal{U}.
\label{t.eq1}\end{equation}
By Proposition~\ref{p.diagonalizable}, there exists a
diagonalizable hyperbolic periodic point $p$ of $f$ such that
\begin{equation}
\frac{1}{\tau(p,f)}\log\sigma(Df^{\tau(p,f)}|E^c)>S_m(f)-\frac{1}{4n}.
\label{t.eq2}\end{equation}
From Proposition \ref{p.estimate} there exists $g\in \mathcal{U}$ and a
hyperbolic basic set $\Lambda(p,g)\subset H(p,g)$ such that
\begin{equation}
h_{top}(g|\Lambda(p,g))>\frac{1}{\tau(p,g)}
\log\sigma(Dg^{\tau(p,g)}|E^c)-\frac{1}{4n}.
\label{t.eq3}
\end{equation}
Therefore, if $\tilde{g}\in \mathcal{U}$ is a diffeomorphism $C^1$-close
to $g$ then there is a continuation of the hyperbolic basic set
$\Lambda(p,g)$ which we denote by $\Lambda(p(\tilde{g}),\tilde{g})$, where
$p(\tilde{g})$ is a continuation of $p$.
Using properties of entropy, (\ref{t.eq3}), continuity of the Lyapounov
exponents, (\ref{t.eq2}), and (\ref{t.eq1}), respectively, we obtain
\begin{align*}
h_{top}(\tilde{g})&\geq h_{top}(\tilde{g}|\Lambda(p(\tilde{g}),\tilde{g}))
= h_{top}(g|\Lambda(p,g))
\geq \frac{1}{\tau(p,g)} \log\sigma(Dg^{\tau(p,g)}|E^c)-\frac{1}{4n}
\\
&\geq \frac{1}{\tau(p,f)} \log\sigma(Df^{\tau(p,f)}|E^c)-\frac{1}{2n}
\geq S_m(f)-\frac{3}{4n} \geq S_m(\tilde{g})-\frac{1}{n}.
\end{align*}
Hence, every $\tilde{g}$ sufficiently $C^1$-close to $g$ belongs to
$\mathcal{B}_{m,n}$.
So, $\mathcal{B}_{m,n}$ contains an open and dense subset of
$\inte(\ph^1(m))$ in view of $\mathcal{R}_m^*$ is dense in
$\inte(\ph^1(m))$.
We denote this subset by $\mathcal{B}_{m,n}^*$

Now, Theorem~\ref{t.trichotomy} implies that
$\inte(\ph^1(1))\cup\dots \cup \inte(\ph^1(d))\cup \mathcal{A}$ is an open
and dense subset of $\dw$.
Hence,
$$
\mathcal{B}_n=\bigcup_{m=1}^d \mathcal{B}_{m,n}^* \cup \mathcal{A},
$$
is an open and dense subset of $\dw$.
Therefore,
$$
\mathcal{R}=\bigcap_{n\in \N} \mathcal{B}_n
$$
is a residual subset in $\dw$, and satisfies the properties required.
In fact, if $f\in \mathcal{R}$ and is non-Anosov, then there exists
$0<m\leq d$ such that $f\in \mathcal{B}_{m,n}$ for every $n>0$.
Hence,
$$
h_{top}(f)\geq S_m(f).
$$
The proof is finished.
\end{proof}

%=========================================================================
\section{Nice properties of strong invariant manifolds}
\label{s.nice_invariant_manifolds}
In this section we obtain properties of strong invariant manifolds
essential to prove Proposition \ref{p.estimate}.
Roughly, we get that for any two hyperbolic periodic points $p$ and
$\tilde{p}$ having strong stable and strong unstable directions well
defined with same dimension, we can perform a symplectic perturbation in
order to obtain a symplectic diffeomorphism such that the continuation of
the hyperbolic periodic point $\tilde{p}$ has a strong quasi-transversal
homoclinic intersection with {\em angle} as close as we want to the
{\em angle} between the strong directions of $p$.

Let us make precise the notion of {\it angle between vector subspaces}.
Given a Riemannian manifold $M$, and vectors $v,w\in T_qM$, we define the
angle between $v$ and $w$ by
$$
\ang(v,w)=\left|\tan\left[\arccos\left(\frac{<v,w>}{\|v\|\|w\|}\right)
\right]\right|.
$$
If $E$ is a vector subspace of $T_qM$, the angle between a vector
$u\in T_qM$  and $E$ is defined by
$$
\ang(v,E)=\min_{w\in E,\, |w|=1}\, \ang(v,w).
$$
Finally if $E,F\subset T_qM$ are subspaces we define
$$
\ang(E,F) =\min_{w\in E,\, |w|=1} \, \ang(w,F).
$$

First of all, we show, after a perturbation, the existence of a hyperbolic
periodic point having strong stable and unstable directions with arbitrary
small angle.

\begin{lemma}
\label{lemma1}
Let $1\leq m\leq d$ and $f\in \inte(\ph^1(m))$.
For any $\eps>0$ and any neighborhood $\mathcal{U}$ of $f$, there exists a
diffeomorphism $g\in \mathcal{U}$ with a hyperbolic periodic point $p$
having $d-m+1$-strong stable and unstable manifolds $W^{ss}_{d-m+1}(p)$
and $W^{uu}_{d-m+1}(p)$ such that $\ang(E_{d-m+1}^{ss}(p),
E_{d-m+1}^{uu}(p))<\eps$.
\end{lemma}

\begin{proof}
The proof is similar in spirit to the proof of Theorem \ref{t.trichotomy}.
Fixed $1\leq m\leq d$, let $f\in \inte(\mathcal{PH}^1_{\omega}(m))$.
Using Proposition~\ref{p.identity}, we can find a diffeomorphism $g$,
$C^1$-close to $f$, having a periodic point $p$ such that
$Dg^{\tau(p,g)}|E^c(p)=Id$.
Now, for any $\eps>0$, we can choose a symplectic basis
$\{e_1, \ldots, e_{2m}\}$ of $E^c(p)$, such that
$\ang(e_1, e_{m+1}) < \eps$.
Hence, taking small constants $\eps_1\ge \eps_2\ge \ldots \ge \eps_m>0$,
we define a symplectic linear map over $E^c(p)$ close to identity, induced
by the symplectic matrix $A=(a_{ij})$ of order $2m \times 2m$, where
$a_{ii}=1-\eps_i$, if $1\leq i\leq m$, $a_{ii}=(1-\eps_i)^{-1}$,
if $m+1\leq i\leq 2m$, and $a_{ij}=0$ if $i\neq j$.

Hence, as done in the proof of Theorem \ref{t.trichotomy}, we can find a
diffeomorphism $h$ $C^1$-close to $g$, such that $p$ still is a periodic
point of $h$, $Dh^{\tau(p)}|E^c(p)=A$ and
$Dh^{\tau(p)}|(E^c(p))^{\omega}=Dg$.
Therefore, $p$ is a hyperbolic periodic point of $h$, and moreover by
choice of $A$ and the symplectic basis of $E^c(p)$, the $n-m+1$-strong
stable and unstable manifolds of $p$ are well defined, with
$e_1\in E_{n-m+1}^{ss}(p)$ and $e_{m+1}\in E_{n-m+1}^{uu}(p)$.
Thus, $\ang(E_{n-m+1}^{ss}(p), E_{n-m+1}^{uu}(p))<\eps$ and the lemma is
proved.
\end{proof}

Next, we state and prove the main result in this section.

\begin{lemma}
\label{lemma2}
Let $f$ be a symplectic diffeomorphism on a $2d$-dimensional symplectic
manifold $M$, with two hyperbolic periodic points $p$ and $\tilde{p}$,
both having $k$-strong stable and unstable manifolds, for some
$1 \leq k\leq d$.
Given $\eps>0$, for any neighborhood $\mathcal{U}$ of $f$ and any
neighborhood $V$ of $\tilde{p}$, there exists a diffeomorphism
$g\in \mathcal{U}$, such that the analytic continuation $p(g)$ of $p$ for
$g$, has a $k$-strong quasi-transversal homoclinic intersection $q$ in
$V$, $q\in W^{ss}_{k}(p(g))\cap W^{uu}_{k}(p(g))$, and moreover
$T_qW^{ss}_{k}(p(g))$ and $T_qW^{uu}_{k}(p(g))$ are $\eps$-close to
$T_{\tilde{p}(g)}W^{ss}_k(\tilde{p}(g))$ and
$T_{\tilde{p}(g)}W^{uu}_k(\tilde{p}(g))$, respectively.
\end{lemma}

\begin{proof}
First, we fix $\eps>0$ small enough.
We may assume $\tilde{p}$ is a fixed point of $f$, replacing $f$ by an
iterate if necessary.
Now, we consider in the neighborhood $V$ of $\tilde{p}$ a continuous
splitting $T_VM = E \oplus F \oplus G$, with $\dim E = \dim G = k$ not
necessarily invariant, which extends the $Df(\tilde{p})$-invariant
partially hyperbolic splitting on $T_{\tilde{p}}M$, i.e. $E_{\tilde{p}} =
E^{ss}_{\tilde{p}}$, $F_{\tilde{p}} = E^c_{\tilde{p}}$, and $G_{\tilde{p}}
= E^{uu}_{\tilde{p}}$.

Now, we consider a neighborhood $\mathcal{U}'\subset \mathcal{U}$ of $f$
and $\delta>0$ given by Lemma~\ref{franks}.
Fixed such $\delta$, there exists $\gamma>0$ such that taking arbitrary
a $k$-dimensional isotropic subspace $G' \subset T_VM $, from
Lemma~\ref{isotropic2} there is a linear $\delta$-perturbation $A$ of the
identity map, such that $\ang (A(G'), E \oplus F)>\gamma$.

With respect to the previous decomposition fixed on $V$ we define the
strong unstable cone fields $C_{\alpha}$ on $V$\;: for $x \in V$
$$
C_{\alpha}(x)=\{w\in T_xM \colon w=w_{cs}+w_u \text{ with }
w_{cs}\in E \oplus F, w_u\in G, \text{ and }
|w_{cs}|\le \alpha |w_u|\}.
$$
We fix $\alpha>0$ such that any vector $v\in T_xM$ satisfying
$\ang(v,E_k^{ss}\oplus E_k^{c})>\gamma$ must belongs to $C_{\alpha}(x)$,
for all $x\in V$.
Now, since the decomposition is partially hyperbolic and
$Df(\tilde{p})$-invariant in $T_{\tilde{p}}M$, there exists $l>0$ such
that by taking smaller neighborhoods $\mathcal{U}''\subset \mathcal{U}'$
and $V'\subset V$ of $f$ and $\tilde{p}$, respectively, for any
$g\in \mathcal{U}''$ and any $x\in \displaystyle\bigcap_{0\leq i\leq l}
g^{-i}(V')$\;:
$$
Dg^{l}(C_\alpha(x))\subset C_{\eps}(g^{l}(x)).
$$
Also, for technical reasons, we can also suppose that the local strong
stable manifold of $\tilde{p}(g)$, $g\in \mathcal{U}''$, has two components
outside the neighborhood $V'$.

From Lemma~\ref{conexao} we can perturb $f$ to find an intersection
between the $k$-dimensional strong unstable manifold of $\tilde{p}$ and
the $k$-dimensional strong stable manifold of $p$.
See Figure a, in Figure~\ref{fig4}.
That is, there exists a diffeomorphism $h\in\mathcal{U}''$, such that there
is $x\in W^{ss}_{k}(p(h))\cap W^{uu}_{k}(\tilde{p}(h))$.
Replacing $x$ by a backward iterate we can suppose $x\in V'$, and since
$h^{-j}(x)$ converges to $\tilde{p}(h)$, after perturbation using
Lemma~\ref{isotropic2} and Lemma~\ref{franks}, if necessary, we can assume
$T_{h^{-j}(x)}(W^{ss}_{k}(p(h)))$ converges to $E^{ss}_{k}(\tilde{p}(h))$,
when $j$ goes to infinity, using partially hyperbolic splitting properties.
Moreover, we can assume the existence of open disks inside
$W^{ss}_{k}(p(h))$ containing $h^{-j}(x)$ converging in the $C^1$-topology
to the local strong stable manifold of $\tilde{p}(h)$, when $j$ goes to
infinity.

Thus, we choose a large positive integer $N_1$ such that
$T_{h^{-N_1}(x)}(W^{ss}_{k}(p(h)))$ is $\eps$-close to
$E^{ss}_{k}(\tilde{p}(h))$, and such that for every $n\geq N_1$ there are
disks $D(n)\subset W^{ss}_k(p(h))$ containing $h^{-n}(x)$ which are
$C^1$-close to the local strong stable manifold of $\tilde{p}(h)$.
See Figure b, in Figure \ref{fig4}.
Also, recall that $h^{-n}(x)$ belongs to $V$, for every $n\geq N_1$, since
$x\in W^{uu}_k(\tilde{p}(h))$.

Since the local strong stable manifold of $\tilde{p}(h)$ has two
components outside $V'$ and since $h$ restricted to this local strong
stable manifold is a contraction, we can take a point
$y\in (D(N_1+l) \cap V')\subset W^{ss}_{k}(p(h))$ such that
$h^{-1}(y)\not\in \cl(V')$ and $h^{j}(y)\in V'$ for $0\leq j\leq l$.
Once again, using Lemma~\ref{conexao} we can find a diffeomorphism
$\tilde{g}\in \mathcal{U}''$ arbitrary $C^1$-close to $h$, and $\tilde{y}$
a point arbitrary close to $y$ such that\;:
\begin{itemize}
\item[$-$] $\tilde{g}^{-1}(\tilde{y})\not\in \cl(V')$ and
$\tilde{g}^{j}(\tilde{y})\in V'$ for $0\leq j\leq l$;
\item[$-$] $\tilde{y}\in
W^{ss}_{k}(p(\tilde{g}))\cap W^{uu}_{k}(p(\tilde{g}))$; and
\item[$-$] $T_{\tilde{g}^l(\tilde{y})}W^{ss}_{k}(p(\tilde{g}))$ is
$\eps$-close to $T_{\tilde{p}(\tilde{g})}W^{ss}_k(\tilde{p}(\tilde{g}))$.
\end{itemize}
Where the last item comes from the continuously variation with respect to
the diffeomorphism of the $k$-strong stable manifold of $p$ and
$\tilde{p}$ in compact parts.

\begin{figure}[!htb]
\vspace{0,5cm}
\centering
\includegraphics[scale=0.9]{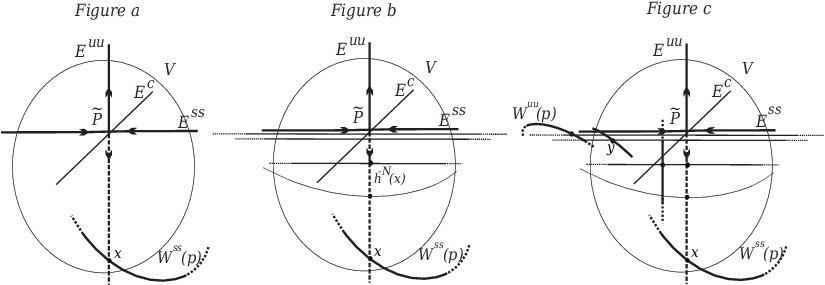}
\caption{\emph{Figure a:} connecting $W^{ss}(p)$ and $W^{uu}(\tilde{p})$;
\emph{Figure b:} taking strong iterated disks of $p$ close to $\tilde{p}$\;;
and \emph{Figure c:} connecting $W^{uu}(p)$ and $W^{ss}(p)$.}
%\vspace{0,5cm}
\label{fig4}
\end{figure}

Let us remark, that for $\tilde{g}$ it is possible that there is no more
strong connection between $W^{ss}_k(\tilde{g}(p))$ and
$W^{uu}_k(\tilde{g}(\tilde{p}))$.
However, as we can see in the remainder of the proof, this is unnecessary.

Considering the isotropic subspace $T_{\tilde{y}}W^{uu}_{k}(p(g))$ in
$T_{\tilde{y}}M$, by the choice of $\delta$ and $\gamma>0$ there exists a
linear symplectic map $\tilde{A}$, $\delta$-close to the identity map,
such that $\ang (\tilde{A}(T_{\tilde{y}}W^{uu}_{k}(p(g))),
E_{\tilde{y}}\oplus F_{\tilde{y}})>\gamma$.
In particular, $A=\tilde{A}\circ D\tilde{g}(\tilde{g}^{-1}(\tilde{y}))$ is
a symplectic linear map $\delta$-close to
$D\tilde{g}(\tilde{g}^{-1}(\tilde{y}))$.
Thus, since $\tilde{g}\in \mathcal{U}''\subset \mathcal{U}'$ and
$\tilde{g}^{-1}(\tilde{y})\notin \cl(V')$ by Lemma~\ref{franks}, we can
find a diffeomorphism $g\in \mathcal{U}$  such that
\begin{itemize}
\item[$-$] $\tilde{y}\in W^{ss}_{k}(p(g))\cap W^{uu}_{k}(p(g)$;
\item[$-$] $T_{g^l(\tilde{y})}W^{ss}_{k}(p(g))$ is $\eps$-close to
$T_{\tilde{p}(g)}W^{ss}_k(\tilde{p}(g))$;
\item[$-$] $g(x)=\tilde{g}(x)$ for every $x\in V'$; and
\item[$-$] $\ang( T_{\tilde{y}}W^{uu}_{k}(p(g)), E_{\tilde{y}}\oplus
F_{\tilde{y}})> \gamma$.
In particular, $T_{\tilde{y}}W^{uu}_{k}(p(g))\subset C_{\alpha}(\tilde{y})$.
\end{itemize}

Thus, since $\tilde{g}\in \mathcal{U}''$ and $g$ coincides with
$\tilde{g}$ on $V'$, by the choice of $l$, we have
$Dg^l(T_{\tilde{y}}W^{uu}_{k}(p(g)))\in C_{\eps}(\tilde{y})$, which
implies that it is $\eps$-close to $E^{uu}_k(\tilde{p}(g))$.
Finally, by continuity of the partially hyperbolic splitting, we also have
$T_{g^l(\tilde{y})}W^{uu}_{k}(p(g))$ is $\eps$-close to
$T_{\tilde{p}(g)}W^{uu}_k(\tilde{p}(g))$.
See Figure c, in Figure \ref{fig4}.
The lemma is proved.
\end{proof}

\section{ Proof of Proposition \ref{p.estimate}}
\label{section5}
To prove Proposition~\ref{p.estimate} we use first Lemmas~\ref{lemma1} and
\ref{lemma2} to perturb a symplectic diffeomorphism $f\in \ph^1(m)$ in
order to find a symplectic $(2d-2m)$-dimensional surface containing a
hyperbolic periodic point $p$ and a segment of strong homoclinic
intersection of $p$.
After, we use arguments in spirit of those ones present in \cite{CT}, to
find a nice hyperbolic set by means of Newhouse's snake perturbations.

\begin{proof}[Proof of Proposition \ref{p.estimate}]
Let $f\in int(\ph^1(m))$ and $p$ be a diagonalizable hyperbolic periodic
point of $f$.
In order to simplify notation, let us suppose that $p$ is a fixed point.

Fixing an arbitrary $\eps>0$, by Lemma~\ref{lemma1}, after a perturbation,
we can suppose there exists a hyperbolic periodic point $\tilde{p}$ of $f$
having $d-m+1$-strong stable and unstable manifolds such that
$\ang(E^{ss}_{d-m+1}(\tilde{p}),E^{uu}_{d-m+1}(\tilde{p}))<\eps/2$.
Thus, since we have defined $(d-m+1)$-strong manifolds for $p$, we can use
Lemma~\ref{lemma2}, to find a diffeomorphism $f_1$ $C^1$-close to $f$,
such that $p(f_1)$ has a $(d-m+1)$-strong quasi-transversal homoclinic
point $q\in W^{ss}_{d-m+1}(p(f_1))\cap W^{uu}_{d-m+1}(p(f_1))$, such that
$$
\ang(T_{q}W^{ss}_{d-m+1}(p(f_1)),T_{q}W^{uu}_{d-m+1}(p(f_1))<
\frac{2\eps}{3}.
$$
Since $f_1$ is arbitrary $C^1$-close to $f$, we can suppose $p(f_1)$
still is a diagonalizable hyperbolic fixed point.

Now, we use a Pasting Lemma of Arbieto and Matheus \cite{AM} to linearize
the diffeomorphism in a small neighborhood $V$ of $p(f_1)$.
More precisely, we can find $f_2$ $C^1$-close to $f_1$ such that
$p(f_1) = p(f_2)$ and $f_2 = Df_1(p)$ in $V$ (in local coordinates).
We remark that after this perturbation, we could have no more a strong
quasi-transversal intersection between $W^{ss}_{d-m+1}(p(f_2))$ and
$W^{uu}_{d-m+1}(p(f_2))$ near $q$.
However, provided that these submanifolds varies continuously in compact
parts with respect to the diffeomorphism, this intersection could be
recovered after a local perturbation of $f_2$ in a neighborhood of $q$.

Let $T_x V = E(x) \oplus F(x) \oplus G(x)$ be a continuous extension (not
necessarily invariant) of the local linear coordinates $\R^{2d}=E^{ss}(p)
\oplus E^c (p) \oplus E^{uu}(p)$ induced by $Df_2(p) = Df_1(p)$, i.e.,
$E(p) = E^{ss}(p)$, $F(p) = E^c(p)$, and $G(p) = E^{uu}(p)$, with $F(x)$
symplectic and $E(x), G(x)$ isotropic.

For $1<m \le d$ we consider $\tilde{E}^c(q) \subset E^c(q)$ a
$(2m-2)$-dimensional symplectic subspace having trivial intersection with
$T_qW^{ss}_{d-m+1} \oplus T_qW^{uu}_{d-m+1}$.
We set $\tilde{E}^c(f^j_2(q)) := Df^j_2(\tilde{E}^c(q)) \subset
E^c(f_2^j(q))$.

We remark now that the local strong stable and unstable manifolds of $p$
coincide with their strong directions restrict to $V$, since $f_2$ is
linear in this neighborhood.
That is, the local strong stable (resp. unstable) manifold of $p$ is
$E^{ss}_{d-m+1}(p)\cap V$ (resp. $E^{uu}_{d-m+1}(p)\cap V$).
Thus, since $q$ is a strong  homoclinic point, for any large positive
integer $k$\;:
$$
f_2^{k}(q)\in E^{ss}_{d-m+1}(p)\cap V \text{ and } f_2^{-k}(q)\in
E^{uu}_{d-m+1}(p)\cap V.
$$

In the reminder of this proof by abuse of notation we denote by $p$ all
its continuations with respect to nearby diffeomorphisms and we denote the
$(d-m+1)$-strong directions and manifolds of $p$ only by $E^*(p)$ and
$W^{*}(p)$, $* = ss, uu$, respectively.
The same for the $(2m-2)$-central direction $E^c(p)$.

\begin{lemma}
There exists a symplectic diffeomorphism $f_3$ $C^1$-close to $f_2$, a
positive integer $K$, a neighborhood $V' \subset V$ of $p$, and small
neighborhoods $U_{-K},\ U_K\subset V'$ of $f_3^{-K}(q)$ and $f_3^{K}(q)$,
respectively, such that
\begin{itemize}
\item $f_3=Df_3(p)=Df_1(p)$ is still linear on $V'$ (in local
coordinates);
\item $f_3^{2K}( (T_{f_{3}^{-K}(q)}W^{ss}(p,f_3) \oplus
T_{f_{3}^{-K}(q)} W^{uu}(p,f_3)) \cap U_{-K}) \subset
(T_{f_{3}^{K}(q)} W^{ss}(p,f_3) \oplus T_{f_{3}^{K}(q)}
W^{uu}(p,f_3) ) \cap U_K$.
\end{itemize}
\label{keypoint}
\end{lemma}

\begin{proof}
First we remark that after a local perturbation, if necessary, we can
suppose that for any large positive integer $k$,
$$
Df_2^{-k}(T_{q}W^{ss}(p))\cap (\tilde{E}^c(f_2^{-k}(q)) \oplus
T_{f_2^{-k}(q)}W^{uu}(p))=0,
$$
and
$$
Df_2^{k}(T_{q}W^{uu}(p))\cap (T_{f_2^{k}(q)}W^{ss}(p) \oplus
\tilde{E}^c(f_2^{k}(q)))=0.
$$
It follows from the dominated splitting properties that
$Df_2^{-k}(T_{q}W^{ss}(p))$  (resp. $Df_2^{k}(T_{q}W^{uu}(p))$ converges
to $E^{ss}(p)$ (resp. $E^{uu}(p)$) when $k$ goes to infinity.
Hence, we can choose a large positive integer $K$ such that both
$Df_2^{-K}(T_{q}W^{ss}(p))$ and $Df_2^{K}(T_{q}W^{uu}(p))$ are close
enough to $E(f_2^{-K}(q))$ and $E(f_2^{K}(q))$, respectively.
Since for any hyperbolic periodic point of a symplectic map its stable
(resp. unstable) manifold is a Lagrangian submanifold, see
Remark~\ref{dec.}, we have that $T_x W^{s}(p)$ (resp. $T_x W^{u}(p)$) is a
Lagrangian subspace for any $x\in W^{s}(p)$ (resp. $W^{u}(p)$).
In particular $T_x W^{ss}(p)$ (resp. $T_x W^{uu}(p)$) is an isotropic
subspace, see Remark~\ref{decomp.}, for any $x\in W^{ss}(p)$ (resp.
$W^{uu}(p)$).

Hence, from Lemma~\ref{isotropic2} there is a symplectic linear map
$\tilde{B}$ on $\R^{2d}$, $C^1$-close to $Id$, such that
$$
\tilde{B}(Df_2^{-K}(T_{q}W^{ss}(p))) = E(f_2^{-K}(q))
\text{ and } \tilde{B}| (\tilde{E}^c(f_2^{-K}(q)) \oplus
T_{f_2^{-K}(q)}W^{uu}(p))) = Id.
$$
Thus, taking $B=Df_2\circ \tilde{B}^{-1}$, we can use Franks Lemma to
perform a local $C^1$-perturbation $f_{2,1}$ of $f_2$ in a neighborhood
$U_{2,1}$ of $f_2^{-K}(q)$, if necessary, such that $f_{2,1} = f_2$ in
$\{f_2^{-K}(q)\} \cup (M \setminus U_{2,1})$ and
$$
Df_{2,1}^{-K}(T_{q}W^{ss}(p))= E(f_{2,1}^{-K}(q)) = E(f_{2}^{-K}(q)).
$$
Moreover, this perturbation does not change the action of $Df_2$ over
$\tilde{E}^c(O(q)) \oplus T_{O(q)}W^{uu}(q)$.

On the other hand, as before, we can perform a local $C^1$-perturbation
$f_{2,2}$ of $f_{2,1}$ in a neighborhood $U_{2,2}$ of $f_{2,1}^{K-1}(q) =
f_2^{K-1}(q)$, if necessary, such that $f_{2,2} = f_{2,1} = f_2$ in
$\{f_2^{K-1}(q)\} \cup (M \setminus U_{2,2})$ and
$$
Df_{2,2}^{2K}(T_{f_2^{-K}(q)}W^{uu}(p)))=
G(f_{2,2}^{K}(q)),
$$
and keeps the action of $Df_{2,1}$ over $\tilde{E}^c(O(q)) \oplus
T_{O(q)}W^{ss}(q,f_{2,1})$.
In particular $E(f_2^{-K}(q)) = T_{f_{2,2}^{-K}(q)} W^{ss}(p)$
We take $V' \subset V$, neighborhood of $p$ such that $f_{2,2}^{-K}(q),
f_{2,2}^{K}(q)\in V'$ and $U_{2,1} \cup U_{2,2} \cap V = \emptyset$.
Thus, $f_{2,2}$ is still linear on $V'$, which implies:
$$
Df_{2,2}^{2K}(T_{f_{2,2}^{-K}(q)} W^{ss}(p) \oplus
T_{f_{2,2}^{-K}(q)} W^{uu}(p))=
T_{f_{2,2}^{K}(q)} W^{ss}(p) \oplus  T_{f_{2,2}^{K}(q)} W^{uu}(p)).
$$
Note that the above perturbations do not change neither the orbit of $q$
nor of $p$.

Finally, we use the Pasting Lemma to get a symplectic diffeomorphism $f_3$
arbitrarily close to $f_{2,2}$ that linearizes $f_{2,2}^{2K}$ in a small
neighborhood $U\subset V'$ of $f_{2,2}^{-K}(q)$ such that
$$
f_3^{2K}((E(f_{3}^{-K}(q)) \oplus T_{f_{3}^{-K}(q)} W^{uu}(p)) \cap U) =
(T_{f_{2}^{K}(q)} W^{ss}(p) \oplus  G(f_3^{K}(q))) \cap V'.
$$

To finish the proof of the lemma we take $U_{-K}=U$ and $U_K=f_3^{2K}(U)$.
\end{proof}

\begin{lemma}
\label{keypoint2}
There is a symplectic diffeomorphism $f_4$, $C^1$-close to $f_3$, such that
$f_4^{2K}(T_{f_{4}^{-K}(q)}W^{ss}(p,f_3)) \cap T_{f_{4}^{K}(q)}
W^{uu}(p,f_4)$ contains a segment of line I.
\end{lemma}

\begin{figure}[!htb]
\vspace{0,5cm}\centering
\includegraphics{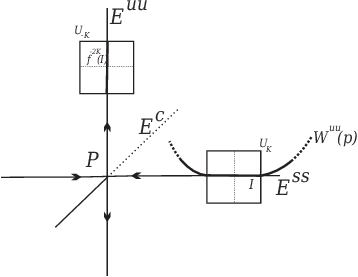}
\caption{Creating a $(2d-2m)$-dimensional surface containing an interval
of strong homoclinic intersections.}
\label{fig6}
\end{figure}

\begin{proof}
Since strong invariant manifolds vary continuously in compact parts
with respect to the diffeomorphism, we have
$\ang(T_qW^{ss}(p,f_{3}), T_qW^{uu}(p,f_{3}))<\eps$ whenever
$f_{3}$ is sufficiently $C^1$-close to $f_1$.
Let $u\in T_{q}W^{ss}(p,f_{3})$ and $w\in T_{q}W^{uu}(p,f_{3})$ be
unit vectors such that $\ang(u,w)<\eps$ and define $E$ and $W$ the
one-dimensional subspaces generated by $u$ and $w$, respectively.
In particular, $E$ and $W$ are near isotropic subspaces.
Also recall that $T_qW^{ss}(p,f_{3}), T_qW^{uu}(p,f_{3})$ is a
symplectic space (see \cite[Section 4]{BV}).

Thus, it follows from Lemma~\ref{isotropic2} there is a symplectic linear
map $A$ defined on $T_q W^{ss}(p,f_{3})\oplus T_qW^{uu}(p,f_{3})$ such
that $A(W)=E$.
Note that $A$ is $C^1$-close to $Id|(T_q W^{ss}(p,f_{3})\oplus
T_qW^{uu}(p,f_{3}))$ whenever $\eps$ is sufficiently small.
By Lemma~\ref{symplectic} there is a symplectic linear map
$L \colon T_qM\rightarrow T_qM$ such that $L|(T_q W^{ss}(p,f_{3})\oplus
T_qW^{uu}(p,f_{3})) = A$ and $L| \tilde{E}^c(q)=Id$.
Then, using Franks Lemma we make a local perturbation of $f_{3}$ in a
neighborhood of $f_{3}^{-1}(q)$ to get a symplectic diffeomorphism $f_4$
$C^1$-close to $f_{3}$ (and so to $f_2$) such that $Df_4(f_4^{-1}(q)) =
L\circ Df_{3}(f_{3}^{-1}(q))$.
Which implies that $T_q W^{ss}(p,f_4)\cap T_qW^{uu}(p,f_4)$ is non
trivial.
We stress that the above local perturbation keeps unchanged the orbits of
$p$ and $q$.

Since $T_q W^{ss}(p,f_4)\cap T_qW^{uu}(p,f_4)$ is non-trivial, there is an
interval of strong homoclinic points, that is $f_4^{2K}(T_{f_{4}^{-K}(q)}
W^{uu}(p,f_4))\cap U) \cap T_{f_{4}^{K}(q)} W^{ss}(p,f_4)$ contains a
segment of line $I$.
The proof is finished.
\end{proof}

Let $f_4$ be $C^1$-close to $f_3$ (and so to $f_2$) as given by
Lemma~\ref{keypoint2}.
After a perturbation, if necessary, we assume that
$T_qW^{ss}(p,f_4)\cap T_qW^{uu}(p,f_4)$ is an one-dimensional subspace.
Hence, if $I$ is sufficiently small this holds for every $x\in I$\;:
$T_xW^{ss}(p,f_4)\cap T_xW^{uu}(p,f_4)$ is an one-dimensional subspace for
every $x\in I$.

What follows is the construction of a local perturbation of the identity
map in $U_K$ to finish the proof of Proposition~\ref{p.estimate}.
In order to simplify notation let us set $\tilde{E}^{ss}(f_4^K(q)) :=
T_{f_4^K(q)}W^{ss}(p,f_3)$, $\tilde{E}^{uu}(f_4^K(q)) :=
T_{f_4^K(q)}W^{uu}(p,f_3)$.
In $U_K$ we consider linear local coordinates given by the extension of
partial hyperbolic decomposition of $\tilde{E}^{ss}(f_4^K(q)) \oplus
\tilde{E}^c(f_4^K(q)) \oplus \tilde{E}^{uu}(f_4^K(q))$.
That is, for every $x\in U_K$ there exists $x^s\in
\tilde{E}^{ss}(f_4^K(q))$, $x^c\in \tilde{E}^c(f_4^K(q))$, and
$x^u\in \tilde{E}^{uu}(f_4^K(q))$ such that
$x=(x^s,x^c, x^u)$.
By construction, in these coordinates $f_4^K(q)=(0,0,0)$.

Now, let $E^s_1 = Df_3^K(E)$ ($E$ as in the proof of Lemma~\ref{keypoint2})
be the one-dimensional subspace of $\tilde{E}^{ss}(f_4^K(q))$
containing the interval $I$ and denote $E^s_2 = Df_3^K(E^s(q))$.
So, $\tilde{E}^{ss}(f_4^K(q))=E^s_1\oplus E^s_2$.
Similarly, we denote $E^u_1 = Df_3^K(W)$ and $E^u_2 = Df_3^K(E^u(q))$.
Thus, $\tilde{E}^{uu}(f_4^K(q)) = E^u_1\oplus E^u_2$ and
$E^s_1 \oplus E^u_1$ is a symplectic subspace.
According to these last direct sums, for $x^s\in \tilde{E}^{ss}
(f_4^K(q))$ and $x^u\in \tilde{E}^{uu}(f_4^K(q))$ we write
$x^s=(x^s_1,x^s_2)$ and $x^u=(x^u_1,x^u_2)$.

Let $N$ be a large positive integer and $\delta>0$ an arbitrary small
real number.
Using the Pasting Lemma of \cite{AM}, we find a symplectic
diffeomorphism $\Theta_N \colon M\rightarrow M$, $\delta$-$C^1$-close to
$Id$, $\Theta_N=Id$ in the complement of $U_K$, and such that, for $r>0$
small enough and $x \in B(f_4(q),r)\subset U_K$,
$$
\Theta_N(x_1^s, x_2^s, x^c,x_1^u, x^u_2) =
\left(x_1^s, x_2^s,\ x^c,\,x^u_1+A\cos\frac{\pi x^s_1 N}{2r},
x^u_2\right),
$$
where $A=\displaystyle\frac{2Rr\delta }{\pi N}$, $R$ a constant
depending only on the symplectic coordinate on $U_K$.
Note that $\Theta_N(I)\cap I$ contains $N$ distinct points and that
$\Theta_N(\tilde{E}^{ss}(f_4^K(q)) \oplus \tilde{E}^{uu}(f_4^K(q))
\cap U_K) \subset \tilde{E}^{ss}(f_4^K(q))\oplus \tilde{E}^{uu}(f_4^K(q))
\cap U_K$.

Now, we use $\Theta_N$ to get a symplectic perturbation of $f_4$.
We set $f_{4,N}=\Theta_N \circ f_4$ which is $\delta$-$C^1$-close to $f_4$
and satisfies
\begin{itemize}
\item $f_{4,N}=f_4$ in the complement of $f_4^{-1}(U_K)$;
\item $f_{4,N}^{2K}\big( (T_{f_{4}^{-K}(q)}W^{ss}(p,f_3) \, \oplus \,
T_{f_{4}^{-K}(q)} W^{uu}(p,f_3)) \cap U_{-K}\big) \subset
\big(T_{f_{4}^{K}(q)} W^{ss}(p,f_3) \\
\oplus \, T_{f_{4}^{K}(q)}
W^{uu}(p,f_3) \big) \cap U_K$;
\item $f_{4,N}^{2K}\big(T_{f_{4}^{-K}(q)} W^{uu}(p,f_4)\cap U_{-K}\big)
\cap (T_{f_{4}^{K}(q)} W^{ss}(p,f_4)\cap U_K)$ contains at least $N$
distinct points.
\end{itemize}
Note that, by construction, $f_{4,N}^{2K}\big(T_{f_{4}^{-K}(q)}
W^{uu}(p,f_4) \cap U_{-K}\subset E^{uu}(p)$ and also $T_{f_{4}^{K}(q)}
W^{ss}(p,f_4) \cap U_K \subset E^{ss}(p)$.
Since we are considering $V'$ in local coordinates where $f_4$ is linear,
these sets belong to $W^{uu}(p)$ and $W^{ss}(p)$ for $f_{4,N}$,
respectively.

This perturbation is a kind of Newhouse's snake perturbation for higher
dimensions, i.e., it destroys the interval of homoclinic intersections and
creates $N$ transversal homoclinic points for $p$ inside $U_K$.
See Figure~\ref{fig7}.
\begin{figure}[!htb]
%\vspace{0,5cm}
\centering{\includegraphics{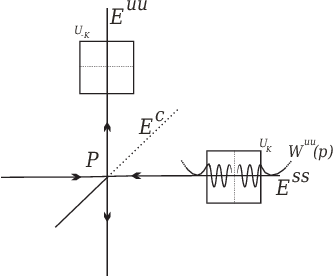}}
\caption{Newhouse's snakes.}
%\vspace{0,5cm}
\label{fig7}
\end{figure}

After we have found these strong homoclinic points in $U_K$, we follow the
arguments developed in step 3 of the proof of  \cite[Proposition~3.1]{CT},
in order to find a hyperbolic set $\Lambda$ satisfying the proposition.
Here $\Lambda\cap V' \subset (E^{ss}(p)\oplus E^{uu}(p))$ and this is a
key point.
For completeness we sketch of how we use the arguments in \cite{CT}.

First, we choose a positive integer $t$ and a rectangle $D_t$ in
$(E^{ss}(p)\oplus E^{uu}(p))\cap U_K$ containing the $N$ transversal
homoclinic points obtained above and also that $g^{t}(D_{t})\cap D_{t}$
has $N$ disjoint connected components.
We also require that $t$ is the smallest possible such that $D_t$ is
$(A/2)$-$C^1$-close to $E^{ss}(p)\cap U_K$ and $f_{4,N}^t(D_t)$ is
$(A/2)$-$C^1$ close to the connected component of $W^{uu}(p)\cap U_K$
containing the $N$ transversal homoclinic points.
Note, $t$ goes to infinity when $N$ goes to infinity.

Therefore, the maximal invariant set in the orbit of $D_t$ is a hyperbolic
set $\Lambda(p,N)$ is conjugated to a product of shift maps, with
topological entropy
$$
h_{top}(f_{4,N}|\Lambda(p,N)) = \displaystyle\frac{1}{t}\log N.
$$

Since the dynamics of $f_{4,N}$ is linear on $V'$ and
$\Lambda(p,N)$ belongs to $V'\subset E^{ss}(p)\oplus E^{uu}(p)$ it follows
straightforward from the proof of \cite[Lemma~4.2]{CT} an upper estimate
for $A$, as follows.

\begin{lemma}
\label{afirma}
For $A$ and $t$ defined as before, there exists a positive integer $K_1$
independent of $A$, such that
$$
A<K_1 \max \{ \|Df_{4,N}^{-t}|E^{uu}(p)\|,\, \|Df_{4,N}^{t}|E^{ss}(p)\|\}.
$$
\end{lemma}

Finally, given a positive integer $k$, we can choose $N$ large enough such
that by the choice of $A$ and Lemma~\ref{afirma}
\begin{equation}
\label{e.min}
\displaystyle\frac{1}{t}\log N >
\min\left\{\frac{1}{t}\log\|Df_{4,N}^{-t}|E^{uu}(p)\|^{-1},\,
\frac{1}{t}\log\|Df_{4,N}^{t}|E^{ss}(p)\|^{-1} \right\} -\frac{1}{2k}.
\end{equation}
Since $f_{4,N}=f_4$ in $V'$ for every $N$, when $t$ goes to infinity the
minimum in \eqref{e.min} goes to the smallest positive Lyapunov exponent
of $p$ restrict to $E^{ss}(p)\oplus E^{uu}(p)$.
As $\dim (E^{ss}\oplus E^{uu}) = 2(d-m+1)$, the smallest positive
Lyapunov exponent of $p$ restrict to this subspace is  equal to
$\log\sigma(Df_{4,N}|E_m^c(p))$.
Thus, taking $g=f_{4,N}$ for $N$ large enough, we have that
$$
h_{top}(g|\Lambda(p,N))> \log\sigma(Dg|E_m^c(p))-\frac{1}{n}.
$$

In the general case, when $p$ is a hyperbolic periodic point, we also can
create a strong homoclinic intersection between $W^{ss}(p)$ and
$W^{uu}(p)$, as before, and thus repeating the above arguments for
$\tilde{f}=f^{\tau(p,f)}$ we can find a nice hyperbolic set
$\tilde{\Lambda}(p,N)$ of a symplectic diffeomorphism
$\tilde{g}=g^{\tau(p,f)}$, such that $g$ is $C^1$-close to $f$, satisfying
$$
h_{top}(\tilde{g}|\tilde{\Lambda}(p,N))> \log\sigma(D\tilde{g}|E_m^c(p))-
\frac{1}{n}.
$$

Therefore the proposition is proved, since the hyperbolic set
$$
\Lambda(p,N)=\bigcup_{0\leq i< \tau(p,f)} g^i(\tilde{\Lambda}(p,N))
$$
of $g$ has topological entropy:
$$
h_{top}(g|\Lambda(p,N))=\frac{1}{\tau(p,f)}
h_{top}(\tilde{g}|\tilde{\Lambda}(p,N)).
$$
The proof is complete.
\end{proof}

%\bibliographystyle{plain}
%\bibliography{bib}

\bigskip

\flushleft

{\bf Thiago Catalan} (tcatalan\@@famat.ufu.br)\\
Faculdade de Matem\'{a}tica, FAMAT/UFU \\
Av. Jo\~ao Naves de Avila, 2121\\
38.408-100, Uberl\^andia,MG, Brazil

\bigskip

\flushleft

{\bf Vanderlei Horita} (vhorita\@@ibilce.unesp.br)\\
Departamento de Matem\'{a}tica, IBILCE/UNESP \\
Rua Crist\'{o}v\~{a}o Colombo 2265\\
15054-000 S. J. Rio Preto, SP, Brazil

\end{document}